\documentclass[11pt]{amsart}
\usepackage[utf8]{inputenc}
\usepackage{amsmath}
\usepackage{amssymb, amsthm, amsfonts, amscd, MnSymbol, url, stmaryrd}
\usepackage[top=1.25in, bottom=1.25in, outer=1.35in, inner=1.35in, heightrounded, marginparwidth=0.9in, marginparsep=.1cm]{geometry}
\usepackage{amsxtra}     
\usepackage{epsfig}
\usepackage{verbatim}
\usepackage{graphicx}
\usepackage[all]{xypic}
\usepackage{todonotes}
\usepackage{marginnote}

\usepackage{mathrsfs}

\usepackage{hyperref}

\usepackage{enumerate}   

\input xy
\xyoption{all}

\newcommand{\xupdownarrow}[1]{%
  {\left\updownarrow\vbox to #1{}\right.\kern-\nulldelimiterspace}
}

\newcommand{\ci}{C^\infty}

\newcommand{\Spec}{\mathrm{Spec}}

\def\bA{\mathbb{A}}
 
\def\bC{\mathbb{C}}

\def\bQ{\mathbb{Q}}

\def\bZ{\mathbb{Z}}

\def\cO{\mathcal{O}}

\def\GL{\operatorname{GL}}
\def\Gal{\operatorname{Gal}}

\newcommand{\Step}{\mathrm{Step}}

\newcommand{\fadd}{\mathcal{A}}

\newcommand{\IC}{\bC}
\newcommand{\ZZ}{\bZ}
\newcommand{\IQ}{\bQ}
\newcommand{\adeles}{\bA}
\newcommand{\isomto}{\overset{\sim}{\rightarrow}}

\newcommand{\ord}{\mathrm{ord}}

\newcommand{\cmfield}{K}

\newcommand{\OK}{\cO_\cmfield}

\newcommand{\CO}{\mathcal{O}}

\newcommand{\Ig}{\mathrm{Ig}}

\newcommand{\Dist}{\mathrm{Dist}}

\newcommand{\Hom}{\mathrm{Hom}}
\newcommand{\Haar}{\mathrm{Haar}}
\newcommand{\Sym}{\mathrm{Sym}}
\newcommand{\fin}{\mathrm{fin}}

\newcommand{\Auniv}{\mathcal{A}_{\rm{univ}}}
\newcommand{\uo}{\underline{\omega}}
\newcommand{\IR}{\mathbb{R}}

\newtheorem{thm}{Theorem}
\numberwithin{thm}{section}
\newtheorem{cor}[thm]{Corollary}

\newtheorem{lem}[thm]{Lemma}
\newtheorem{prop}[thm]{Proposition}

\newtheorem{question}{Question}

\theoremstyle{definition}

\newtheorem{defi}[thm]{Definition}

\newtheorem{example}[thm]{Example}
 
\theoremstyle{remark}

\newtheorem{rmk}[thm]{Remark}

\numberwithin{equation}{section}

\title[Eisenstein measures]{An introduction to Eisenstein measures}

\author[E. E. Eischen]{E. E. Eischen}
\thanks{Partially supported by NSF Grants DMS-1559609 and DMS-1751281.}

\address{E. Eischen\\
Department of Mathematics\\
University of Oregon\\
Fenton Hall\\
Eugene, OR 97403-1222\\
USA}
\email{eeischen@uoregon.edu}

\begin{document}
\bibliographystyle{amsalpha}

\newpage
\setcounter{page}{1}
\maketitle
\vspace{-0.2in}
\begin{abstract}
This paper provides an introduction to Eisenstein measures, a powerful tool for constructing certain $p$-adic $L$-functions.  First seen in Serre's realization of $p$-adic Dedekind zeta functions associated to totally real fields, Eisenstein measures provide a way to extend the style of congruences Kummer observed for values of the Riemann zeta function (so-called {\em Kummer congruences}) to certain other $L$-functions.  In addition to tracing key developments, we discuss some challenges that arise in more general settings, concluding with some that remain open.
\end{abstract}

\section{Introduction}\label{sec:intro}
In the mid 1800s, Kummer proved that the values of the Riemann zeta function at negative odd numbers satisfy striking congruences modulo powers of any prime number $p$.  More precisely, he proved that if $k$ and $k'$ are positive even integers not divisible by $p-1$, then for all positive integers $d$,
\begin{align*}
\left(1-p^{k-1}\right)\zeta(1-k)\equiv \left(1-p^{k'-1}\right)\zeta\left(1-k'\right) \mod p^d,
\end{align*}
whenever $k\equiv k'\mod \varphi\left(p^d\right)$, with $\varphi$ denoting Euler's totient function \cite{kummer}.  (By Euler's work a century earlier, for all positive integers $k$, $\zeta(1-k)$ is the rational number $(-1)^{k+1}\frac{B_k}{k}$ for all positive integers $k$, where $B_k$ denotes the $k$-th Bernoulli number, defined as the coefficients in the Taylor series expansion $
\frac{te^t}{e^t-1} = \sum_{n=0}^\infty B_n\frac{t^n}{n!}$.)  Kummer's motivation for studying these congruences stemmed from his interest in determining when a prime is (what we now call) {\it regular}, i.e. does not divide the class number of the cyclotomic field $\IQ\left(\zeta_p\right)$ with $\zeta_p\in \IC^\times$ a primitive $p$-th root of unity, in which case Kummer could prove Fermat's Last Theorem for exponents divisible by $p$.  As part of his investigations, Kummer had shown that the regularity is equivalent to a condition on the values of the Riemann zeta function.  More precisely, Kummer showed that a prime $p$ is regular if and only if if does not divide the numerator of the Bernoulli numbers $B_2, B_4, \ldots, B_{p-3}$ \cite{kummer36, kummer34}.
 
After Kummer's exciting discoveries, this topic then lay nearly dormant for a century.  Even though Hensel (who had been one of Kummer's graduate students) introduced the $p$-adic numbers soon after Kummer's death \cite{hensel}, the first formulation of a $p$-adic zeta function (a $p$-adic analytic function that interpolates the values of a modified zeta function at certain points, essentially encoding Kummer's congruences) occurred only in the 1960s, as a result of work of Kubota and Leopoldt \cite{KL}.  These functions, too, play a key role in the structure of cyclotomic fields, on a deeper level than Kummer had originally developed (or presumably even envisioned, given mathematical developments during the century following Kummer's life).

In the 1960s, Iwasawa linked the behavior of Galois modules over towers of cyclotomic fields to $p$-adic zeta-functions, forming the foundations of {\it Iwasawa theory}, a $p$-adic theory for studying families of arithmetic data \cite{iw, iw2}.  For example, the subgroup $\Gamma:=\Gal\left(\IQ_\infty/\IQ\right)\cong\ZZ_p$ of $\Gal(\IQ(\mu_{p^\infty})/\IQ)$ acts on the $p$-part of the ideal class group of the $p^m$-th cyclotomic extension for each $m\geq1$.  The inverse limit (under a norm map) of these groups is a module over the Iwasawa algebra $\Lambda:=\Lambda_{\Gamma, \ZZ_p}=\ZZ_p[\![\Gamma]\!]\cong \ZZ_p[\![T]\!]$.   
The main conjecture of Iwasawa theory, proved in \cite{MW}, says a realization $\theta\in\Lambda$ of a $p$-adic $L$-function generates the characteristic ideal of this $\Lambda$-module.  Thus, the $p$-adic $L$-function controls substantial structural information about this collection of class groups.

Iwasawa's conjectures were further generalized.  In particular, R. Greenberg predicted the existence of more general $p$-adic $L$-functions ($p$-adic analytic functions that can be realized as elements of certain Iwasawa algebras and whose values encode analogues of Kummer's congruences for more general $L$-functions) and their meaning in the context of certain Galois modules.
In particular, the Greenberg--Iwasawa main conjectures \cite{green3, green2, green1} predict that for a wide class of ordinary Galois representations $\rho$, there is a $p$-adic $L$-function $\mathscr{L}_\rho$ interpolating values of an $L$-function associated to $\rho\otimes\chi$ as $\chi$ varies over certain Hecke characters and that $\mathscr{L}_\rho$ can be realized as the generator of the characteristic ideal of a certain $\Lambda$-module (a {\it Selmer} group), where $\Lambda:=\CO[\![\Gamma_K]\!]$, with $\Gamma_K$ the Galois group of a compositum of $\ZZ_p$-extensions of $K$ and $\CO$ an appropriate $p$-adic ring.  In other words, the main conjectures predict $p$-adic $L$-functions govern the structure of Selmer groups as Galois modules.

Given $L$-functions' starring role in the main conjectures of Iwasawa theory (and their conjectured existence, including not only in Greenberg's conjectures, but also in, for example, \cite{CoPR, coatesII}), it is natural to ask:
\begin{question}\label{guiding-question}
Given an $L$-function whose values at certain points are known to be algebraic, how might we construct a $p$-adic $L$-function encoding congruences between values of (a suitably modifed at $p$) version of that $L$-function?
\end{question}
The main goal of this paper is to introduce particular tools, {\it Eisenstein measures}, which have proved to be especially useful for constructing $p$-adic $L$-functions during the past half-century, at least under certain conditions.  Even putting aside the challenge of trying to prove main conjectures in Iwasawa theory, it is generally hard to answer Question \ref{guiding-question}.  One might first look to Kummer or to Kubota--Leopoldt (who actually considered $p$-adic Dirichlet $L$-functions) for answers, in the hope that earlier techniques could be generalized.  This would, however, require extending congruences coming from Bernoulli polynomials to other settings, and unfortunately, we do not generally have realizations of values of $L$-functions in terms of similarly convenient polynomials.  While there has been some successful work in that direction (see, e.g., work of Barsky \cite{barsky} and P. Cassou-Nogues \cite{cassou-nogues}, who employed formulas of Shintani that have recently been further explored in work of Charollois--Dasgupta \cite{Ch-Da}), one of the most powerful tools for constructing $p$-adic $L$-functions in increasing generality during the past half-century comes from the theory of {\it $p$-adic modular forms}.

In the early 1970s, Serre produced the first $p$-adic families of Eisenstein series (the first instances of {\it Eisenstein measures}, which arose as part of his development of the theory of $p$-adic modular forms) and used them (together with Iwasawa's construction of the $p$-adic zeta function as an element of an Iwasawa algebra) to 
construct $p$-adic Dedekind zeta functions associated to totally real number fields
\cite{serre}.  
Because modular forms are special cases of automorphic forms and because the behavior of $L$-functions is closely tied to the behavior of automorphic forms (at least, in certain settings), this approach seemed more amenable to generalization.
 Indeed, its promise was immediately realized, including by Coates--Sinnott \cite{coates-sinnott}, Deligne--Ribet \cite{DR}, and Katz \cite{kaCM}, who employed Eisenstein measures in constructions of $p$-adic $L$-functions associated to Hecke characters for quadratic real fields, real number fields, and CM fields (with Katz proving the CM case only for half of all primes, a restriction that stood for over four decades until work of Andreatta--Iovita in 2019 \cite{andreatta-iovita}), respectively.  

Given that Eisenstein series govern key properties of certain $L$-functions (not only algebraicity, but also functional equations and meromorphic continuation), it is perhaps not surprising that they play key roles in our context as well.  Thus, another important question becomes:
\begin{question}\label{guiding-question2}
How might we construct $p$-adic families of Eisenstein series or, more specifically, $p$-adic Eisenstein measures?
\end{question}
Constructing $p$-adic Eisenstein measures is generally hard.  Were it not for the prestigious journal in which all those first papers following Serre's introduction of Eisenstein measures were published or the accomplished mathematicians whose names are attached to these results, the reader could not be blamed for thinking these results sound incremental.  Instead, though, this should be viewed as evidence that seemingly small tweaks to the data to which  $L$-functions are attached can lead to significant technical challenges in constructing the corresponding $p$-adic $L$-functions.

For most readers of this paper, implicit in Question \ref{guiding-question2} is that we want Eisenstein series that can be directly related to values of $L$-functions.  It is worth noting, though, that interest in Question \ref{guiding-question2} extends beyond number theory.  At least in the cases of modular forms and automorphic forms on unitary groups of signature $(1,n)$, $p$-adic families of Eisenstein series also are of interest in homotopy theory \cite{hopkinsICM, hopkins94, beh}.

Returning to Question \eqref{guiding-question}, we note the favorable fact that (at least, as it appears to this author) all known constructions of $p$-adic $L$-functions seem to rely on building on the specific techniques employed in the proof of the algebraicity of the values of the $\IC$-valued $L$-function in question.  Thus, if you know a proof of algebraicity, then you are at least in possession of clues to the techniques needed to construct $p$-adic $L$-functions.  For example, Serre's development of the theory of $p$-adic modular forms and its use for constructing $p$-adic zeta functions built on the approaches of Klingen and Siegel (who were, in turn, building on ideas of Hecke, as recounted in \cite[Section 1.3]{BCG} and \cite{IO}) for studying algebraicity of values of zeta functions by exploiting properties of Fourier coefficients of modular forms \cite{klingen, siegel1, siegel2}.  Likewise, Katz's construction of $p$-adic $L$-functions associated to Hecke characters of CM fields employs Damerell's formula, which was first used by Shimura to prove algebraicity.  In a similar spirit, Hida's construction of $p$-adic Rankin--Selberg $L$-functions attached to modular forms builds on Shimura's proof of algebraicity via the Rankin--Selberg convolution \cite{shimura-RS, hi85}.  In a more recent instance, the construction of $p$-adic $L$-functions for unitary groups due to the author, Harris, Li, and Skinner \cite{HELS} employs the doubling method (a pull-back method of the sort used by Shimura to prove algebraicity, e.g. in \cite{shar}, and which specializes in the case of rank $1$ unitary groups to Damerell's formula).

\subsection{Organization of this paper}
Now that we have established some history and motivation for studying Eisenstein series, we spend the remainder introducing their mathematical features.  Section \ref{sec:meas} introduces distributions and measures from several viewpoints, each of which is useful for different aspects of working with $p$-adic $L$-functions.  Section \ref{sec:serre} then discusses the first example of an Eisenstein measure, produced by Serre as a tool for constructing $p$-adic Dedekind zeta functions associated to totally real fields.  This development inspired efforts to construct Eisenstein measures valued in the space of $p$-adic Hilbert modular forms, tools for constructing $p$-adic $L$-functions attached to certain Hecke characters, as discussed in Section \ref{sec:hilbert}.  We conclude with a discussion of generalizations to other $L$-functions (Section \ref{sec:general}), as well as some of the significant challenges encountered as one tries to construct useful Eisenstein measures.

\subsection{Acknowledgements}
I would like to thank the local organizers of Iwasawa 2019, Denis Benois and Pierre Parent, for inviting me to give the four lectures that eventually led to this paper, as well as for their patience as I wrote it.  I would also like to thank them, along with the scientific organizers of Iwasawa 2019 (Henri Darmon, Ming-Lun Hsieh, Masato Kurihara, Otmar Venjakob, and Sarah Zerbes), for organizing an exciting, educational workshop.  The many excellent discussions I had with participants after each of my lectures influenced my approach to explaining the material in this paper.  I would especially like to thank Chi-Yun Hsu and Sheng-Chi Shih for taking careful notes in my lectures and sharing them with me.  In addition, I would like to thank Pierre Charollois for alerting me to some interesting aspects of the history of the approach of using constant terms of modular forms to study zeta functions.  I am also grateful to the referee for providing helpful feedback.

\section{$p$-adic distributions and measures}\label{sec:meas}

Motivated by Iwasawa's and Greenberg's conjectures about the Galois theoretic role of $p$-adic $L$-functions, our goal is to find an element in an Iwasawa algebra whose values at certain characters encode congruences between certain (modified) $L$-functions.  Distributions and measures will provide a convenient tool for realizing $p$-adic $L$-functions inside Iwasawa algebras.

For a more detailed treatment of distributions and measures, we especially recommend \cite[\S 7]{MSD} and \cite[\S 7.1-7.2 and \S 12.1-12.2]{wa}.

\subsection{Conventions and preliminaries}\label{conventions-section}
Throughout this paper, we fix a prime number $p$.  For convenience, we assume $p$ is odd.
We denote by $\IC_p$ the completion of an algebraic closure of $\IQ_p$.  We call a ring $\CO$ a {\em $p$-adic ring} if it is complete and separated with respect to the $p$-adic topology, i.e. $\CO \cong \varprojlim \CO/p^n \CO$.  Given a $p$-adic ring $\CO$ and a profinite group $G = \varprojlim_{n}G/G_n$ with each $G_n$ a finite index subgroup of $G$, we define the Iwasawa algebra
\begin{align*}
\Lambda_G:=\CO[\![G]\!]:=\varprojlim_n\CO\left[G/G_n\right].
\end{align*}
Following the usual conventions in Iwasawa theory, we define
\begin{align*}
\Gamma &:= 1+p\ZZ_p\\
\Lambda&:=\ZZ_p[\![\Gamma]\!]\cong\ZZ_p[\![T]\!].
\end{align*}
We also denote by $\mu_{p-1}$ the
multiplicative subgroup of order $p-1$ in $\ZZ_p^\times$.  Given a number field $F$, we denote by $F\left(p^\infty\right)$ the maximal abelian unramified away from $p$ extension of $F$.

In addition, throughout this section, let $T$ be a locally compact totally disconnected topological space, and let $W$ be an abelian group.  (For example, $T$ could be the Galois group of a $\ZZ_p$-extension and $W$ could be the ring of integers in a finite extension of $\IQ_p$.)  We denote by $\Step(T)$ the group of $\ZZ$-valued functions on $T$ that are locally constant of compact support.  For any compact open subset $U\subseteq T$, we denote by $\chi_U\in \Step(T)$ the characteristic function of $U$.

\subsection{Distributions}
We begin by introducing distributions, which include measures as a special case.
\begin{defi}
A {\em distribution} on $T$ with values in $W$ is a homomorphism 
\begin{align*}
\mu: \Step (T)&\rightarrow W.
\end{align*}
We set the notation
\begin{align*}
\int_T\varphi(t)d\mu:=\mu(\varphi)
\end{align*}
for each $\varphi\in\Step(T)$.
\end{defi}
The space of $W$-valued distributions on $T$ is then
\begin{align*}
\Dist(T, W):=\Hom(\Step(T), W).
\end{align*}
Observe that we have a bijection between $\Dist(T, W)$ and the set $\fadd(T, W)$ of finitely additive $W$-valued functions on compact open subsets of $T$.  By abuse of notation, given $\mu\in\Dist(T, W)$, we also denote by $\mu\in\fadd(T, w)$ the corresponding element under this bijection.  More precisely, given a compact open subset $U\subset T$,
\begin{align*}
\mu(U) := \int_Ud\mu:=\mu\left(\chi_U\right)=\int_T\chi_Ud\mu.
\end{align*}

\subsubsection{Distributions on (pro)finite sets}\label{profinite-distribution-section}
Observe that if $T$ is finite, then since $T$ is totally disconnected, $\Dist(T, W)$ is identified with the abelian group of $W$-valued functions on $T$.  So if $X$ is the inverse limit of a collection of finite sets $X_i$, $i\in I$ a directed poset, such that whenever $i\geq j$, we have a surjection
\begin{align*}
\pi_{ij}: X_i\twoheadrightarrow X_j
\end{align*}
and whenever $i\geq j\geq k$, $\pi_{jk}\circ \pi_{ij} = \pi_{ik}$, then we can reformulate the notion of $W$-valued distribution on $X$ as a collection of $W$-valued maps
\begin{align*}
\mu_j: X_j\rightarrow W
\end{align*}
such that
\begin{align*}
\mu_j(x) = \sum_{\left\{y|\pi_{ij}(y) = x\right\}} \mu_i(y)
\end{align*}
for all $i\geq j$ and all $x\in X_j$.  So we have
\begin{align*}
\Dist(X, W) = \varprojlim_n\Dist\left(X_n, W\right).
\end{align*}

\subsection{Measures}
We now suppose that $W$ is a finite-dimensional Banach space over an extension $K$ of $\IQ_p$, as this case will be particularly interesting to us.  
\begin{defi}
A {\em $W$-valued measure} on $T$ is a bounded $W$-valued distribution on $T$.
\end{defi}
\begin{defi}
If a measure $\mu$ takes values in a subgroup $A\subseteq W$, then we call $\mu$ an {\it $A$-valued measure}.
\end{defi}

Given topological spaces $X$ and $Y$, we denote by $\mathcal{C}(X, Y)$ the space of continuous maps from $X$ to $Y$.
Observe that if $T$ is compact and $W$ is a finite-dimensional $K$-Banach space, then there is a bijection
\begin{align*}
\left\{W\mbox{-valued measures on } T\right\}\leftrightarrow \left\{\mbox{bounded homomorphisms of $K$-Banach spaces $\mathcal{C}(T, K)\rightarrow W$}\right\}.
\end{align*}

Likewise, if $\CO$ is a $p$-adically complete ring, then we have bijections
\begin{align*}
\left\{\CO\mbox{-valued measures on } Y\right\}\leftrightarrow\left\{\CO\mbox{-linear maps } \mathcal{C}(Y, \CO)\rightarrow \CO\right\}\leftrightarrow\left\{\ZZ_p\mbox{-linear maps } \mathcal{C}\left(Y, \ZZ_p\right)\rightarrow \CO\right\}.
\end{align*}
More generally, given an $\CO$-valued measure $\mu$ on $Y$ and a homomorphism $\varphi: \CO\rightarrow \CO'$, with $\CO'$ also a $p$-adic ring, we get an $\CO'$-valued measure $\mu'$ on $Y$ defined by
\begin{align*}
\int_Y fd\mu' := \varphi\left(\int_Y f d\mu\right)
\end{align*}
for all $f\in \mathcal{C}(Y, \ZZ_p)$.

\subsubsection{Measures on profinite groups}\label{sec:pfgm}
Our main case of interest is the case where $T$ is a profinite group.  We write $T = \varprojlim_j T/T_j$ with the subgroups $T_j$ the ones that are open for the topology on $T$. (So the groups $T_j$ are the finite index, normal subgroups of $T$.)  Then for any $p$-adic ring $\CO$, we have an isomorphism of $\CO$-modules
\begin{align}\label{dictionary-iso}
\psi: \Dist\left(T, \CO\right)&\isomto \Lambda_T=\CO[\![T]\!]\\
\mu&\mapsto \alpha_\mu:= \left(\sum_{T/T_j}\mu_j(g)g\right)_{j\geq 0},
\end{align}
where $\mu_j$ is as in Section \ref{profinite-distribution-section}.  

Note that each $f\in \mathcal{C}(T, \CO)$ can be extended $\CO$-linearly to a function on the group ring $\CO[T]$, via 
\begin{align}\label{polydef}
f(\sum_{g\in T}a_g g) = \sum_{g\in T}a_g f(g)
\end{align}
for each finite sum $\sum_{g\in T}a_g g\in \CO[T]$ with $a_g\in\CO$. (Since $\CO[T]$ is a group ring, $a_g=0$ for all but finitely many $g$.)  Also, note that $\CO[T]$ is a subring of $\CO[\![T]\!]$, via
\begin{align*}
\CO[T]&\hookrightarrow \CO[\![T]\!]\\
\sum_{g\in T}a_g g&\mapsto \left(\sum_{T/T_j}a_g (g\bmod T_j)\right)_{j\geq 0}
\end{align*}
Since $\CO[T]$ is dense in $\CO[\![T]\!]$, we extend the map in Equation \eqref{polydef} continuously to $\CO[\![T]\!]$.  (The rings $\CO[T/T_j]$ are endowed with the product topology coming from $\CO$, and so the same is true for $\CO[\![T]\!]$.)  Of particular interest is the case where $T$ is generated by a topological generator $\gamma$ (e.g. $\gamma=1+p$ in  $T=1+p\ZZ_p$) and $f: T\rightarrow \CO^\times$ is a group homomorphism, in which case the elements of $\CO[\![T]\!]$ can be identified with power series $\sum_j a_j \gamma^j$, and $f(\sum_j a_j \gamma^j) = \sum_j a_j f(\gamma)^j.$  Similarly, if $T = \Gamma_1\times \cdots\times \Gamma_d$ with $\Gamma_i= 1+p\ZZ_p$, $i=1, \ldots, d$, with generators $\gamma_i=1+p\ZZ_p$, respectively, then each element of $ \CO[\![T]\!]$ can be expressed as a power series in $\gamma_1, \ldots, \gamma_d$, and for $h=\sum_{n_1, \ldots,n_d\geq 0}a_{n_1, \ldots, n_d}\gamma_1^{n_1}\cdots\gamma_d^{n_d}\in\CO[\![T]\!]$, $f(h)=\sum_{n_1, \ldots,n_d\geq 0}a_{n_1, \ldots, n_d}f\left(\gamma_1\right)^{n_1}\cdots f\left(\gamma_d\right)^{n_d}\in\CO[\![T]\!]$.
 
The inverse map $\psi^{-1}$ is given by
\begin{align*}
\mu_\alpha \mapsfrom \alpha,
\end{align*}
where for each $f\in\mathcal{C}(T, \CO)$
\begin{align*}
\mu_\alpha(f):= f(\alpha).
\end{align*}
If $\CO$ is flat over $\ZZ_p$, then each element $\alpha\in \Lambda_T$ corresponding to a measure $\mu_\alpha$ is completely determined by $\int_T\chi d\mu_\alpha$, where $\chi$ varies over finite order characters with values in extensions of $\IQ_p$ (see Proposition \ref{prop:half}).

\subsubsection{First examples}
Let $K$ be a finite extension of $\IQ_p$, and let $T$ be an infinite profinite group.
\begin{example}
It is a simple exercise to show that the $K$-valued {\em Haar distributions} $\mu_{\Haar}$ (i.e. translation invariant distributions, so $\mu_{\Haar}(U+y) = \mu_{\Haar}(U)$ for all $y\in T$ and compact open subsets $U\subset T$) on $T$ are not measures if $T$ is a pro-$p$ group (but are measures if $T$ is a pro-$\ell$ group with $\ell\neq p$).
\end{example}

\begin{example}
Fix an element $g\in T$.  The {\em Dirac distribution} $\delta_g$ defined by 
\begin{align*}
\delta_g(U) := \begin{cases}
1 & \mbox{ if $g\in U$}\\
0 & \mbox{ else},
\end{cases}
\end{align*}
for all compact open subsets of $T$, is a measure on $T$.  Under the isomorphism $\psi$ in \eqref{dictionary-iso}, $\mu_g$ corresponds to the element $g\in \Lambda_T$.
\end{example}

\subsection{Bernoulli distributions and Dirichlet $L$-functions}
We now briefly introduce a measure that produces a $p$-adic Dirichlet $L$-function.  More details are available in \cite[\S 12.2]{wa}.

Given a Dirichlet character $\chi$ be a Dirichlet character, let $L(s, \chi)$ be the associated Dirichlet $L$-function.  Then for all positive integers $n$,
\begin{align*}
L(1-n, \chi) = -\frac{B_{n, \chi}}{n},
\end{align*}
where the numbers $B_{n, \chi}$ are the generalized Bernoulli numbers, i.e. the numbers defined by
\begin{align*}
\sum_{a=1}^f\frac{\chi(a)te^{at}}{e^{ft}-1} = \sum_{n=0}^\infty B_{n, \chi}\frac{t^n}{n!},
\end{align*}
with $f$ denoting the conductor of $\chi$.  We also define a modified Dirichlet $L$-function $L^{(p)}(1-n, \chi) = \left(1-\chi(p)p^{n-1}\right)L\left(1-n, \chi\right)$.
If $\chi$ is the trivial character, so $f = 1$, then $B_{n, \chi} = B_n$, where $B_n$ denotes the $n$-th Bernoulli number, and $L(\chi, 1-n) = \zeta(1-n)$ is the Riemann zeta function studied by Kummer in the mid-1800s, as discussed in Section \ref{sec:intro}.  

\begin{rmk}
As above, let $n$ be a positive integer.  Note that when $\chi$ is odd and $n$ is positive, $L(1-n, \chi)=0$.  Likewise, when $\chi$ is even and $n$ is odd, $L(1-n, \chi)=0$, unless $\chi$ is the trivial character and $n=1$ (in which case we obtain $\zeta(0) =-\frac{1}{2}$).  This can be seen from the functional equation for $L(s, \chi)$, as explained in, e.g., \cite[Chapter 4]{wa}.
\end{rmk}

Let
\begin{align*}
\omega: \ZZ_p^\times\rightarrow\mu_{p-1}\subseteq \ZZ_p^\times
\end{align*}
denote the Teichm\"uller character (so $\omega(a)\equiv a\mod p$ for each $a\in\ZZ_p^\times$).  Kummer's congruences are a special case of the following:
\begin{thm}[\cite{KL}]\label{thm:KL}
Let $\chi$ be a Dirichlet character.  Then there exists a $p$-adic meromorphic (analytic, if $\chi\neq 1$) function $L_p\left(1-n, \chi\right)$ on $\left\{x\in\IC_p\mid |s|_p<p^{\frac{p-2}{p-1}}\right\}$ such that
\begin{align*}
L_p\left(1-n, \chi\right) = \left(1-\chi\omega^{-n}(p)p^{n-1}\right)\frac{-B_{n, \chi\omega^{-n}}}{n} = L^{(p)}(1-n, \chi\omega^{-n})
\end{align*}
for all positive integers $n$.
\end{thm}

This is the first example of a {\it $p$-adic $L$-function}, i.e. a $p$-adic analytic function whose values at certain points agree with values of (suitably modified) $\IC$-valued $L$-functions.

\begin{rmk}
Fix an integer $n_0$ such that $0<n_0<p-1$.  Then for all $n\equiv n_0\mod p-1$, $\omega^{n} = \omega^{n_0}$ and 
\begin{align*}
L_p\left(1-n, \omega^{n_0}\right) = \zeta^{(p)}(1-n).
\end{align*}
Thus, we easily locate the values studied by Kummer among those $p$-adically interpolated by $L_p$.  Furthermore, by expressing $L_p(s, \chi)$ in terms of a $\ZZ_p$-valued measure on $\ZZ_p^\times$ (for example, by setting $d=1$ in Equation \eqref{equ:ec} below), we recover the congruences of Kummer (as also noted in \cite[Corollary 12.3]{wa}).
\end{rmk}

For each nonnegative integer $n$, let $B_n(X)$ denote the $n$-th Bernoulli polynomial, i.e. the polynomial defined by
\begin{align*}
\frac{te^{Xt}}{e^t-1} = \sum_{n=0}^\infty B_n(X) \frac{t^n}{n!}.
\end{align*}
So 
\begin{align*}
B_n(0) &= B_n \\
B_n(1) &= \begin{cases}
B_n & \mbox{ if $n\neq 1$}\\
B_1+1 & \mbox{ if $n = 1$}
\end{cases}
\end{align*}
If $\chi$ is a Dirichlet character, $f$ is the conductor of $\chi$, and $F$ is a positive integer divisible by $f$, then
\begin{align*}
B_{n, \chi} = F^{n-1}\sum_{a=1}^F\chi(a)B_n\left(\frac{a}{F}\right).
\end{align*}
We also have $B_k(X) = \sum_{i=0}^k \begin{pmatrix}k\\ i\end{pmatrix} B_i X^{k-i}$ and $B_k(1-X) = (-1)^k B_k(X)$ for all nonnegative integers $k$.

For each positive integer $i$, we define
\begin{align*}
Y_i:=\frac{1}{i}\ZZ/\ZZ,
\end{align*}
and for all positive integers $j\divides i$, we define
\begin{align*}
\pi_{ij}: Y_i\rightarrow Y_j\\
y\mapsto \frac{i}{j}\times y.
\end{align*}

\begin{defi}
The {\em $k$-th Bernoulli distribution} is the distribution $\phi = \left(\phi_i\right)_{i\geq 1}$ on $Y := \varprojlim_i Y_i$ defined for each positive integer $i$ by
\begin{align*}
\phi_i\left(\frac{a}{i}\right):=i^{k-1}B_k\left(\left\{\frac{a}{i}\right\}\right).
\end{align*}
\end{defi}

While the Bernoulli distribution is not a measure, we modify it to obtain a measure on $X := \varprojlim_n X_n$, where $X_n := \left(\ZZ/dp^{n+1}\right)^\times$ and $d$ is a fixed integer, as follows.  Fix $c\in \ZZ$ such that $\gcd(c, dp) = 1$.  For $x_n\in X_n$, we define
\begin{align*}
E_c\left(x_n\right):=E_{c, 1}\left(x_n\right) = B_1\left(\left\{\frac{x_n}{dp^{n+1}}\right\}\right) - B_1\left(\left\{\frac{c^{-1}x_n}{dp^{n+1}}\right\}\right) + \frac{c-1}{2},
\end{align*}
where $\left\{\right\}$ denotes the fractional part of a number.  Then, as further discussed in the proof of \cite[Theorem 12.2]{wa}, $E_c$ is a $\ZZ_p$-valued measure, and furthermore, letting $\langle\cdot \rangle$ denote the projection onto $1+p\ZZ_p$, we have
\begin{align}\label{equ:ec}
\int_{\left(\ZZ/dp\ZZ\right)^\times\times\left(1+p\ZZ_p\right)}\chi\omega^{-1}\langle\rangle^s dE_c = -\left(1-\chi(c)\langle c\rangle^{s+1}\right)L_p\left(s, \chi\right),
\end{align}
for all Dirichlet characters $\chi$ of conductor $dp^m$ with $m$ a nonnegative integer and $s\in\ZZ_p$.

\subsection{Some convenient spaces for defining measures}\label{sec:conven}
When constructing more general measures, in particular Eisenstein measures, it will be convenient to establish some particular subsets of characters on which it is sufficient to define a measure in order for that measure to be uniquely determined.  More precisely, we have the following.

For each of the following two lemmas, let $\CO$ be a $p$-adic ring.
\begin{lem}\label{meas-finite}
An $\CO$-valued measure on $\Gamma = 1+p\ZZ_p$ is completely determined by its values on characters of finite order.
\end{lem}

\begin{lem}\label{meas-infinite}
An $\CO$-valued measure on $\Gamma = 1+p\ZZ_p$ is completely determined by its values on $\langle \rangle^k: \Gamma\rightarrow \Gamma\subset \ZZ_p$ for any infinite set of $k\in\ZZ_p$.
\end{lem}

\begin{proof}[Proof of Lemmas \ref{meas-finite} and \ref{meas-infinite}]
Note that $\CO[\![\Gamma]\!]$ is isomorphic to the power series ring $\CO[\![T]\!]$ (which, as explained in \cite[Section 7.1]{wa}, follows from the isomorphisms $\CO[\Gamma_n]\cong\CO[T]/\left((1+T)^{p^n}-1\right)$, $\gamma\leftrightarrow 1+T$ where $\Gamma_n = \Gamma/\Gamma^{p^n}$). The proofs of both lemmas then follow from the Weierstrass Preparation Theorem, which tells us that if $0\neq f(T)\in \CO[\![T]\!]$, then $f(T) = \pi^rP(T) U(T)$, with $\pi$ a non-unit in $\CO$, $r$ a nonnegative integer, $P(T)$ a monic polynomial whose non-leading coefficients are all divisible by $\pi$, and $U(T)\in\CO[\![T]\!]^\times$ \cite[Chapter VII Section 4]{bourbaki1998commutative}.  Consequently, each nonzero element of $\CO[\![T]\!]$ has only finitely many zeroes $\beta\in \IC_p$ with $|\beta|_p<1$ (since a polynomial can have only finitely many zeroes, and an element of $\CO[\![T]\!]^\times$ cannot have any zeroes with absolute value $<1$).

Now, let $\mu$ be an $\CO$-valued measure on $\Gamma$, and let $f_\mu\in\CO[\![T]\!]$ be the power series corresponding to $\mu$ as in Section \ref{sec:pfgm}.  Then $\mu(\chi) = f_\mu\left(\chi(\gamma)-1\right)$.  So if $\mu$ vanishes at infinitely many finite order characters or infinitely many $\langle \rangle^k$, then $\mu$ is identically $0$.
\end{proof}

More generally, we have the abstract {\em Kummer congruences}, a generalization of the style of congruences established by Kummer.

\begin{thm}[Abstract Kummer congruences]\label{KC-thm}
Let $Y$ be a compact, totally disconnected space, let $\CO$ be a $p$-adic ring that is flat over $\ZZ_p$, and let $I$ be some indexing set.  Let $\left\{f_i\right\}_{i\in I}\subseteq \mathcal{C}(Y, \CO)$ be such that the $\CO\left[\frac{1}{p}\right]$-span of the functions $f_i$ is uniformly dense in $\mathcal{C}\left(Y, \CO\left[\frac{1}{p}\right]\right)$.  Let $\left\{a_i\right\}_{i\in I}\subseteq \CO$.  Then there exists an $\CO$-valued $p$-adic measure $\mu$ on $Y$ such that
\begin{align*}
\int_Yf_i = a_i
\end{align*}
for all $i\in I$ if and only if the elements $a_i$ satisfy the {\em abstract Kummer congruences}, i.e.:

{\em Given $\left\{b_i\right\}_{i\in I}\subset \CO\left[\frac{1}{p}\right]$ such that $b_i=0$ for all but finitely many $i\in I$, together with a nonnegative integer $n$ such that
$\sum_{i\in I} b_if_i(y)\in p^n\CO$ for all $y\in Y$, we have $\sum_{i\in I}b_ia_i\in p^n\CO$.
}
\end{thm}
\begin{proof}
This is \cite[Proposition (4.0.6)]{kaCM}, which is proved in {\em loc. cit.}
\end{proof}

When working with a profinite abelian group, the following consequence is particularly convenient for constructing Eisenstein measures in general.

\begin{prop}[First half of Proposition (4.1.2) of \cite{kaCM}]\label{prop:half}
Let $G$ be a profinite abelian group.  Let $\CO$ be a $p$-adically complete ring that is flat over $\ZZ_p$, and suppose that $R$ contains a primitive $n$-th root of unity for all $n$ such that $G$ contains a subgroup of index $n$.  Let $\mu$ be an $\CO$-valued $p$-adic measure on $G$, and let $\chi_0$ be a continuous homomorphism from $G$ to $\CO^\times$.  Then $\mu$ is completely determined by the values $\int_G\chi_0\chi d\mu$ as $\chi$ ranges over finite order characters of $G$.
\end{prop}

\subsubsection{Dictionary between several approaches to defining $p$-adic measures}

We conclude this section with Figure \ref{fig:dictionary}, which summarizes the connections between several formulations of the definition of an $\CO$-valued $p$-adic measure given above, each of which is useful for different aspects of constructing $p$-adic $L$-functions.  In the figure, $\CO$ is a $p$-adic ring, and $G = \varprojlim_i G_i$ is a profinite $p$-adic group, with transition maps $\pi_{ij}: G_i\rightarrow G_j$ whenever $i\geq j$ and $\pi_{jk}\circ \pi_{ij} = \pi_{ik}$ for all $i\geq j\geq k$.

\begin{figure}
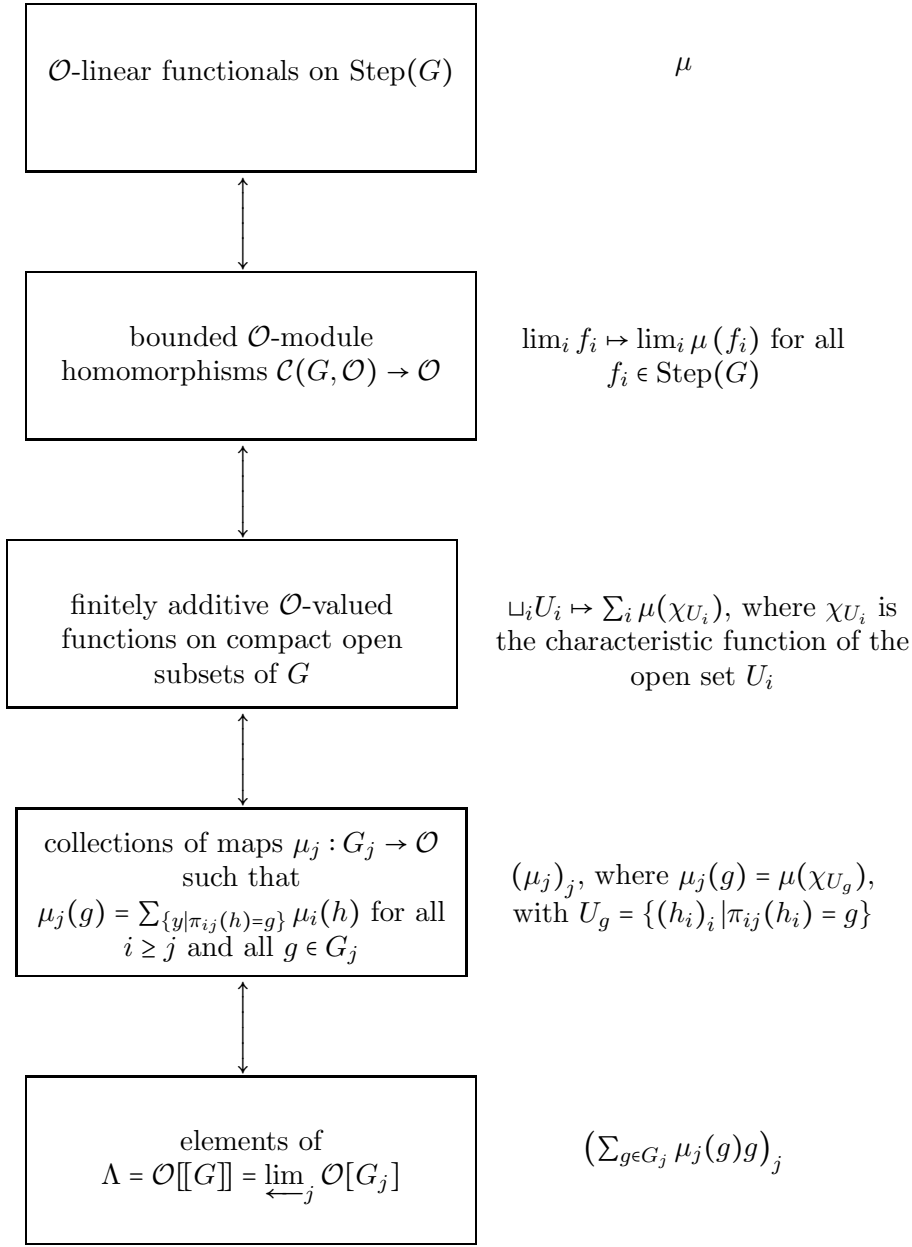

\begin{center}

\framebox{
\parbox[t][2.0cm]{5.50cm}{

\addvspace{0.6cm} \centering 

$\CO$-linear functionals on $\Step(G)$

} 
}\parbox[t][2.0cm]{5.50cm}{

\addvspace{0.6cm} \centering 

$\mu$
}

\hspace{-2.25in}$\xupdownarrow{0.75cm}$

\framebox{
\parbox[t][2.0cm]{5.50cm}{

\addvspace{0.6cm} \centering

bounded $\CO$-module homomorphisms $\mathcal{C}(G, \CO)\rightarrow \CO$

} 
}\parbox[t][2.0cm]{5.50cm}{

\addvspace{0.6cm} \centering 

$\lim_i f_i\mapsto \lim_i\mu\left(f_i\right)$ for all $f_i\in \Step(G)$
}

\hspace{-2.25in}$\xupdownarrow{0.75cm}$

\framebox{
\parbox[t][2.0cm]{5.50cm}{

\addvspace{0.6cm} \centering 

finitely additive $\CO$-valued functions on compact open subsets of $G$

} 
}\hspace{0.2in}\parbox[t][2.0cm]{5.50cm}{

\addvspace{0.6cm} \centering 

$\sqcup_i U_i\mapsto \sum_i\mu(\chi_{U_i}),$
where $\chi_{U_i}$ is the characteristic function of the open set $U_i$
}

\hspace{-2.25in}$\xupdownarrow{0.75cm}$

\framebox{
\parbox[t][2.0cm]{5.50cm}{

\addvspace{0.2cm} \centering 

collections of maps $\mu_j: G_j\rightarrow \CO$ such that $\mu_j(g)= \sum_{\left\{y|\pi_{ij}(h) = g\right\}} \mu_i(h)$
for all $i\geq j$ and all $g\in G_j$

} 
}\hspace{0.1in}\parbox[t][2.0cm]{5.50cm}{

\addvspace{0.6cm} \centering 

$\left(\mu_j\right)_j$, where $\mu_j (g) = \mu(\chi_{U_g}),$ with $U_g = \{\left(h_i\right)_i| \pi_{ij}(h_i) = g\}$ 
} 

\hspace{-2.25in}$\xupdownarrow{0.75cm}$

\framebox{
\parbox[t][2.0cm]{5.50cm}{

\addvspace{0.6cm} \centering 

elements of $\Lambda = \CO[\![G]\!] = \varprojlim_j\CO[G_j]$

} 
}\parbox[t][2.0cm]{5.50cm}{

\addvspace{0.6cm} \centering 

$\left(\sum_{g\in G_j}\mu_j(g)g\right)_j$
}

\end{center}
\caption{\label{fig:dictionary} Dictionary between several formulations of $p$-adic measures}
\end{figure}

\section{A first look at $p$-adic Eisenstein measures}\label{sec:serre}

We are primarily interested in measures as a vehicle for obtaining $p$-adic $L$-functions inside an Iwasawa algebra.  While Bernoulli numbers were useful constructing the measure in Equation \eqref{equ:ec}, they do not necessarily generalize to many other $L$-functions of interest.  It turns out that $p$-adic modular forms provide a convenient tool for constructing $p$-adic $L$-functions in much more generality, while also producing $p$-adic Dedekind zeta functions associated to totally real fields.  

\begin{rmk}
Because of their links with $L$-functions, we will be particularly interested in {\em Eisenstein measures}, measures whose values on certain sets of characters (like those in Section \ref{sec:conven}) are Eisenstein series.
\end{rmk}

In this section, we briefly introduce $p$-adic modular forms, following Serre's approach.  For more details, see \cite{serre}.  We denote by $v_p$ the valuation on $\IQ_p$ such that $v_p(p) = 1$.  For $f(q) = \sum_{n=0}^\infty a_nq^n\in\IQ_p[\![q]\!]$, we define
\begin{align*}
v_p(f) := \inf_n v_p\left(a_n\right).
\end{align*}
So $v_p(f)\geq m$ if and only if $f\equiv 0\mod p^m$ and $v_p(f)\geq 0$ if and only if $f\in\ZZ_p[\![q]\!]$.  Let $\left\{f_i\right\}\subseteq\IQ_p[\![q]\!]$.  We write $f_i\rightarrow f$ and say ``The sequence $f_1, f_2, \ldots$ converges to $f$'' if $v_p\left(f_i-f\right)\rightarrow \infty$, i.e. the coefficients of $f_i$ converge uniformly to those of $f$ as $i\rightarrow \infty$.  We also write $f\equiv g\mod p^m$ if $v_p(f-g)\geq m.$

\begin{example}
Let $k\geq 4$ be an even integer.  Consider the level $1$, weight $k$ Eisenstein series $G_k$ whose Fourier expansion is given by
\begin{align*}
G_k (z) &= \frac{\zeta(1-k)}{2} + \sum_{n\geq 1}\sigma_{k-1}(n) q^n
\end{align*}
where $q = e^{2\pi i z}$ and $\sigma_{k-1}(n) = \sum_{d\divides n}d^{k-1}.$
In the 1800s, Kummer proved that if $p-1\ndivides k$, then $\frac{\zeta(1-k)}{2} $ is $p$-integral, as well as that if $k\equiv k'\mod p-1$, then $\frac{\zeta(1-k)}{2} \equiv \frac{\zeta(1-k')}{2} \mod p$ \cite{kummer}.  So if we also apply Fermat's little theorem to the non-constant coefficients, we see that
\begin{align*}
G_k\equiv G_{k'}\mod p
\end{align*}
whenever $k\equiv k'\not\equiv 0 \mod p-1$.
\end{example}

\subsection{Congruences $\mod p^m$}
Recall that, for convenience, we assume $p$ is odd.  The reader who is curious about $p=2$ can find the analogous statements for that case in \cite{serre}.
\begin{thm}[TH\'EOR\`EME 1 of \cite{serre}]\label{weights-thm}
Let $m\in\ZZ_{\geq 1}$, and let $f, g\in\IQ[\![q]\!]$ be modular forms of weights $k, k',$ respectively, with $v_p\left(f-g\right)\geq v_p(f)+m$.  If $f\neq 0$, then $k\equiv k'\mod (p-1)p^{m-1}$.
\end{thm}

Ultimately, we want not just congruences but $p$-adic measures, which leads us to Section \ref{sec:pmforms}.

\subsection{$p$-adic modular forms}\label{sec:pmforms}

Let $X_m= \ZZ/(p-1)p^{m-1}\ZZ$, and let $X = \varprojlim_m X_m = \ZZ_p\times \ZZ/(p-1)\ZZ$.  We identify $X$ with the space of $\ZZ_p^\times$-valued characters of $\ZZ_p^\times=\left(\ZZ/p\ZZ\right)^\times\times\left(1+p\ZZ_p\right)$, i.e.
\begin{align*}
X& =\ZZ_p\times \ZZ/(p-1)\ZZ\\
k&\leftrightarrow\left(s, u\right)
\end{align*}
corresponds to the $\ZZ_p^\times$-valued character of $\ZZ_p^\times$ defined by
\begin{align*}
a\mapsto a^k:=\langle a\rangle^s\omega^u(a).
\end{align*}

\begin{defi}[Serre] A {\em $p$-adic modular form} is a power series $f = \sum_{n\geq 0}a_n q^n\in\IQ_p[\![q]\!]$ such that there exists a sequence of modular forms $f_1, f_2, \ldots$ such that $f_i\rightarrow f$.
\end{defi}

As a consequence of Theorem \ref{weights-thm}, we see that a nonzero $p$-adic modular form $f = \lim_i f_i$ has weight $k=\lim_i k_i\in X$, where $k_i$ denotes the weight of $f_i$.  A $p$-adic limit of $p$-adic modular forms is again a $p$-adic modular form $f$, and if $f\neq 0$, the weights again converge as in Theorem \ref{weights-thm}.

\begin{cor}[COROLLAIRE 1 of \cite{serre}]\label{coro1}
Let $f = \sum_{n\geq 0}a_nq^n$ be a $p$-adic modular form of weight $k\in X$.  Suppose the image of $k$ in $X_{m+1}$ is nonzero.  Then $v_p\left(a_0\right)+m\geq \inf_{n\geq 1}v_p\left(a_n\right)$.
\end{cor}
\begin{proof}  We briefly recall Serre's proof.
If $a_0=0$, then the corollary is immediate.  Suppose now that $a_0\neq 0$.  Let $g = a_0$, so $g$ is a modular form of weight $k'=0$, and
\begin{align*}
v_p(f-g) = \inf_{n\geq 1}v_p\left(a_n\right).
\end{align*}
Also, since the image of $k$ in $X_{m+1}$ is nonzero, $k\not\equiv k'$ in $X_{m+1}$.  So by Theorem \ref{weights-thm},
\begin{align*}
v_p(f-g) < v_p(g)+m+1.
\end{align*}
Consequently,
\begin{align*}
v_p\left(a_0\right)+m+1>\inf_{n\geq 1}v_p\left(a_n\right),
\end{align*}
so
\begin{align*}
v_p\left(a_0\right)+m\geq \inf_{n\geq 1}v_p\left(a_n\right).
\end{align*}
\end{proof}

\begin{cor}[COROLLAIRE 2 of \cite{serre}]\label{conv-cor}
Consider $p$-adic modular forms $f^{(i)} = \sum_{n=0}^\infty a_n^{(i)}q^m$ of weights $k^{(i)}$, for $i=1, 2, \ldots$, respectively.  Suppose that both of the following hold:
\begin{align*}
\varinjlim_i a_n^{(i)} &= a_n\in\IQ_p \mbox{ for all } n\geq 1\\
\varinjlim k^{(i)} &= k\in X, \mbox{ with } k\neq 0.
\end{align*}
Then $a_0^{(i)}$ converges $p$-adically to an element $a_0\in\IQ_p$, and $f= \sum_{n=0}^\infty a_nq^n$ is a $p$-adic modular form of weight $k$.
\end{cor}

\begin{example}[Application to $G_k$]\label{appGk}
Applying Corollary \ref{conv-cor} to a sequence of Eisenstein series $G_{k_i}$, $i=1, 2, \ldots$, with $k_i\geq 4$ and even for all $i$ and such that $k_i\rightarrow\infty$ in the archimedean metric and also $k_i\rightarrow k\in X$, we obtain a $p$-adic modular form (in fact, a {\em $p$-adic Eisenstein series}, i.e. a $p$-adic limit of Eisenstein series)
\begin{align*}
G_k^\ast:=G_k^{(p)}:=\varinjlim_iG_{k_i}=\frac{\zeta^\ast\left(1-k\right)}{2}+\sum_{n\geq 1}\sigma_{k-1}^\ast(n)q^n
\end{align*}
with $\sigma_{k-1}^\ast(n) := \sigma_{k-1}^{(p)}(n):=\sum_{\substack{d\divides n\\ p\ndivides d}}d^{k-1}$ and $\zeta^\ast(1-k):=\lim_{i\rightarrow \infty}\zeta\left(1-k_i\right)$.
\end{example}

Since the $p$-adic number $\zeta^\ast\left(1-k\right)$ is a $p$-adic limit of values of the Riemann zeta function, it is natural to ask about its relationship to values of the Kubota--Leopoldt $p$-adic zeta function.  This is given in Theorem \ref{thm:se3} below.  More generally, a consequence of Theorem \ref{fs-thm} is the construction of $p$-adic Dedekind zeta functions as elements of $\Lambda$ (i.e. as $p$-adic measures).

We say that an element of $k=(s, u)\in X = \ZZ_p\times \ZZ/(p-1)\ZZ$ is {\em even} if $k\in 2X$ (equivalently, since we are assuming $p$ is odd, $u\in 2\ZZ/(p-1)\ZZ$).  Otherwise, we say $(s, u)$ is {\em odd}.  

\begin{thm}[TH\'EOR\`EME 3 of \cite{serre}]\label{thm:se3}
If $(s, u)\neq 1$ is odd, then $\zeta^\ast(s, u) = L_p\left(s, \omega^{1-u}\right)$, where $L_p$ denotes the Kubota--Leopoldt $p$-adic zeta function from Theorem \ref{thm:KL}.
\end{thm}
\begin{proof}  We recall Serre's proof.
If $\zeta'$ denotes the function
\begin{align*}
(s, u)\mapsto L_p\left(s, \omega^{1-u}\right),
\end{align*}
then $\zeta'$ is the Kubota--Leopoldt $p$-adic zeta function, and $\zeta'(1-k) =\left(1-p^{k-1}\right)\zeta(1-k)$ for each positive even integer $k$. 

If $k\in 2X$ (so $1-k$ is odd), $k_i\rightarrow k$ in $X$, $k_i\rightarrow\infty$ in the archimedean topology, then
\begin{align*}
\zeta'(1-k) &= \lim_{i\rightarrow\infty}\zeta'\left(1-k_i\right) = \lim_{i\rightarrow \infty}\left(1-p^{k_i-1}\right)\zeta\left(1-k_i\right)\\
& = \lim_{i\rightarrow \infty}\zeta\left(1-k_i\right) = \zeta^\ast(1-k).
\end{align*}
\end{proof}

For any fixed even $u\neq 0$ in $\ZZ/(p-1)\ZZ$, we can also prove that the function $s\mapsto \zeta^\ast(1-s, 1-u)$ arises as an element of $\Lambda:=\Lambda_\Gamma$, without reference to the work of Kubota--Leopoldt (and also without reference to the work of Iwasawa, who proved this as well).  
First, we note that for each positive integer $m\nequiv 0 \mod p-1$, the rational numbers $\zeta(1-m) = (-1)^{m+1}\frac{B_m}{m}$ were already known in the mid-1800s to be $p$-integral, thanks to the von Staudt--Clausen theorem \cite{staudt, clausen} and a result of von Staudt on numerators of Bernoulli numbers (\cite{vstaudt}, which was later rediscovered and misattributed, as discussed on \cite[p. 136]{girstmair}).  So the $p$-adic limits $\zeta^\ast(1-s, 1-u)$ are elements of $\ZZ_p$ whenever $u\neq 0$ in $\ZZ/(p-1)\ZZ$.  (Alternatively, Corollary \ref{coro1} shows that because all the higher order Fourier coefficients of the Eisenstein series $G_{(s, u)}^\ast$ are $p$-integral, so is the constant term of $G_{(s, u)}^\ast$.)  Applying Corollary \ref{conv-cor} to the $p$-adic Eisenstein series $G_{(s, u)}^\ast$ from Example \ref{appGk}, we obtain congruences for the constant terms $\zeta^\ast(1-s, 1-u)$.  So applying Lemma \ref{meas-infinite} and Theorem \ref{KC-thm}, we see that for fixed even $u\neq 0$ in $\ZZ/(p-1)\ZZ$, $\zeta^\ast(1-s, 1-u)$ can be obtained as an element of $\Gamma$.  Via different methods, Iwasawa's work also addressed the case where $u=0$ \cite{iw2}.  We summarize these results in Theorem \ref{iwasawathm}.

\begin{thm}\label{iwasawathm} Let $\Lambda:=\Lambda_\Gamma$, and fix an even $u\in \ZZ/(p-1)\ZZ$.  Then:
\begin{enumerate}
\item{If $u\neq 0$, then the function $\langle\rangle^s\mapsto \zeta^\ast(1-s, 1-u)$ is an element of $\Lambda:=\Lambda_\Gamma$.}\label{myproof}
\item{The function $s\mapsto \zeta^\ast (1-s, 1)^{-1}$ is an element $\Lambda\cong \ZZ_p[\![T]\!]$, and moreover, is of the form $Tg(T)$ with $g(T)$ invertible in $\ZZ_p[\![T]\!]$.}\label{iproof}
\end{enumerate}
\end{thm}

For $n\geq 1$ the $n$th Fourier coefficient of the $p$-adic Eisenstein series $G_{k}^\ast$, with $k=(s,u)$, is of the form $\sigma_{k-1}^\ast(n)=\sum_{\substack{d\divides n\\ p\ndivides d}}d^{-1}\omega(d)^k\langle d\rangle^k=\sum_{\substack{d\divides n\\ p\ndivides d}}d^{-1}\omega(d)^u\langle d\rangle^s$, which gives an element of $\Lambda$, when we fix $u$.  Consequently, for fixed $u\neq 0$, the coefficients of $G_{(s,u)}^\ast$ can be viewed as elements of $\Lambda$ (by Theorem \ref{iwasawathm}, Part \eqref{myproof}), and furthermore, for $u=0$, the coefficients of the normalized Eisenstein series $E_{s}^\ast := (\zeta^\ast(1-s, 1)/2)^{-1}G_{(s, 0)}^{\ast}$ can be viewed as elements of $\Lambda$ (by Theorem \ref{iwasawathm}, Part \eqref{iproof}).

More generally, we have the following result.

\begin{thm}[TH\'EOR\`EME 17 and TH\'EOR\`EME 18 of \cite{serre}]\label{fs-thm}
Let $f_s$ be a $p$-adic modular form of weight $k(s) = (sr, u)\neq 0$ for some fixed $r$ and $u$.  Suppose the function $\langle \rangle^s\mapsto a_n\left(f_s\right)$ is in $\Lambda:=\Lambda_\Gamma$ for all $n\geq 1$.
\begin{enumerate}
\item{If $u\neq 0$ in $\ZZ/(p-1)\ZZ$, then the same is true for $n=0$.}\label{fs-thm1}
\item{If $u = 0$ in $\ZZ/(p-1)\ZZ$, then $\langle\rangle^s\mapsto \zeta^\ast(1-rs, 1)^{-1}a_0\left(f_s\right)$ is in $\Lambda$.}\label{fs-thm2}
\end{enumerate}
\end{thm}
\begin{proof}

Serre's proof of each part of this theorem involves a careful analysis of the element $f_s':=f_sE_{-rs}^\ast\in \Lambda$, which is of weight $(0, u)$ and has the same constant term as $f_s$.  Since we will not need the details in this paper, we do not elaborate here and instead refer the reader to \cite[proofs of Theorems 17 and 18]{serre}.  

We note, however, that we can also give an alternate proof for Part \eqref{fs-thm1}, i.e. when $u\neq 0$, using the results developed thus far in the present paper:  If $a_n\left(f_s\right)$ is in $\Lambda$ for all $n\geq 1$, then Corollaries \ref{coro1} (with $m=0$) and \ref{conv-cor} guarantee the constant terms meet the conditions necessary to apply Lemma \ref{meas-infinite} and Theorem \ref{KC-thm}, so we can realize the constant term $a_0\left(f_s\right)$ as an element of $\Lambda$.
\end{proof}

Serre uses Theorem \ref{fs-thm} to obtain a $p$-adic Dedekind zeta function $\zeta_K^\ast$, for $K$ a totally real number field, as an element of $\Lambda$ (where $\zeta_K^\ast$ is defined analogously to $\zeta^\ast$ and occurs as the constant term of an Eisenstein series).

\section{Hilbert modular forms and $L$-functions attached to Hecke characters}\label{sec:hilbert}

Serre's use of $p$-adic families of Eisenstein series to construct $p$-adic zeta functions inspired constructions in other contexts.  We now summarize a generalization to the space of $p$-adic Hilbert modular forms, where realizations of Eisenstein measures enabled the construction of $p$-adic $L$-functions attached to Hecke characters of totally real or CM fields.

\subsection{The strategy of Deligne--Ribet}
Our goal now is to introduce the strategy of Deligne--Ribet from \cite{DR} to $p$-adically interpolate values of $L(s,  \rho)$ for $\rho$ a finite order Hecke character of a totally real field $K$ unramified away from $p$.  Note that for negative integers $s$, $L(s, \rho)$ lies in the field extension $\IQ(\rho)$ obtained by adjoining all values of $\rho$ to $\IQ$. Following the conventions established in Section \ref{conventions-section}, for any number field $F$, we denote by $F\left(p^\infty\right)$ the maximal abelian extension of $F$ that is unramified away from $p$.

\begin{thm}[Main Theorem (8.2) of \cite{DR}]\label{DR-thm}
Fix a totally real field $K$ and a prime-to-$p$ ideal $\mathfrak{A}$ of $K$.  Then there exists a $\ZZ_p$-valued $p$-adic measure $\mu_{\mathfrak{A}}$ on $G:=\Gal\left(K\left(p^\infty\right)/K\right)$ such that for all positive integers $k$ and finite order characters $\rho$ on $G$,
\begin{align*}
\int_G\rho\cdot\mathbf{N}^k d\mu_\alpha = \left(1-\rho(\mathfrak{A})\mathbf{N}\mathfrak{A}^{k+1}\right)L^{(p)}\left(-k, \rho\right),
\end{align*}
where $\mathbf{N}$ denotes the norm and $L^{(p)}(-k, \rho) = \prod_{\mathfrak{p}\divides p}(1-\rho(\mathfrak{p})\mathbf{N}(\mathfrak{p})^{k}) L(-k, \rho)$.
\end{thm}

To prove Theorem \ref{DR-thm}, Deligne and Ribet work in the space of Hilbert modular forms.  As one might expect from Serre's approach to constructing the $p$-adic zeta function, one step toward proving Theorem \ref{DR-thm} is the construction of Eisenstein series (this time, in the space of Hilbert modular forms) of weight $k$ with $L(1-k, \rho)$ as the constant term, for each positive integer $k$.  Similarly to the Eisenstein series $G_k^\ast$, it is easy to see the non-constant terms of the Eisenstein series in \cite{DR} satisfy congruences as the weight $k$ varies.  
Now that we are in the setting of Hilbert modular forms, though, we need a new approach to proving that the constant terms satisfy congruences.  This requires the theory of $p$-adic Hilbert modular forms and $q$-expansion principles, which require more geometry than the discussion thus far.

\subsubsection{Ingredients from the theory of $p$-adic Hilbert modular forms}\label{sec:pHMF}
We briefly delve into the setup of $p$-adic Hilbert modular forms, the space where the families of Eisenstein series from \cite{kaCM, DR} live.  For more details, see \cite[Chapter 4]{hida}, \cite[Chapter I]{kaCM}, or \cite{goren}.  

We can give a formulation of Hilbert modular forms as sections of line bundles over a moduli space $\mathcal{M}$ of Hilbert--Blumenthal abelian varieties (with additional structure).  More precisely, fix a totally real number field $K$ of degree $g$ over $\IQ$, a fractional ideal $\mathfrak{c}$ of $K$, and an integer $N\geq 4$ prime to $p$.  Let $\OK$ denote the ring of integers in $K$, and let $\mathfrak{d}^{-1}$ denote its inverse different.  There is a scheme $\mathcal{M}:=M\left(N, \mathfrak{c}\right)$ over $\Spec\left(\OK\right)$ classifying triples $(X, i, \lambda)$, consisting of an abelian scheme $X$ of relative dimension $g$ together with an action of $\OK$ on it, a level structure $i: \mathfrak{d}^{-1}\otimes_{\ZZ}\mu_N\hookrightarrow X $, and a $\mathfrak{c}$-polarization $\lambda: X^\vee\isomto X\otimes_{\OK}\mathfrak{c}$ (where $X^\vee$ denotes the dual abelian scheme to $X$).

We denote by $\pi:\Auniv\rightarrow \mathcal{M}$ the universal object, and we define $\uo:=\pi_\ast\Omega^1_{\Auniv/\mathcal{M}}.$  
  The space of Hilbert modular forms of weight $\left(k(\sigma)\right)_{\sigma: K\hookrightarrow \IR}$ is identified with $H^0\left(\mathcal{M}, \boxtimes_\sigma\uo\left(k_\sigma\right)\right)$.  Note that $\uo\left(k_\sigma\right)$ is a subsheaf of $\Sym^k(\uo)$, where $k = \sum_\sigma k(\sigma)$.  Note that there exists a smooth toroidal compactification $\overline{\mathcal{M}}$ of $\mathcal{M}$ that includes the cusps of $\mathcal{M},$ and the universal abelian scheme $\Auniv$ extends to the universal semi-abelian scheme over $\overline{\mathcal{M}}$.  Also note that when $[K:\IR]>1,$ K\"ocher's principle guarantees that a Hilbert modular form over $\mathcal{M}$ extends holomorphically to the cusps.  As explained in \cite[Example 5.3]{DR}, \cite[Section 4.1]{hida} (see also \cite[Section 1.1]{kaCM} on algebraic $q$-expansions), when we work over a $\IQ$-algebra $R$ (for example, $R=\IC$), the cusps are in bijection with fractional ideals $\mathfrak{A}$ of $K$ (which we will call the ``cusp corresponding to $\mathfrak{A}$''). 

\begin{rmk}
While we shall not need this fact here (as we are working in settings specific to Katz and Deligne--Ribet), it is worth noting that as discussed in \cite[p. 2-3]{AIP2}, the notion of ``Hilbert modular form'' for $F\neq \IQ$ varies slightly depending on where in the literature one looks.  More precisely, the moduli problem represented by $\mathcal{M}=M\left(N, \mathfrak{c}\right)$ corresponds to the group $G^\ast = G\times_{\mathrm{Res}_{K/\IQ}\mathbb{G}_m} \mathbb{G}_m$, where $G=\mathrm{Res}_{K/\IQ}\GL_2$, $G\rightarrow \mathrm{Res}_{K/\IQ}\mathbb{G}_m$ is the determinant morphism, and $\mathbb{G}_m\rightarrow \mathrm{Res}_{K/\IQ}\mathbb{G}_m$ is the diagonal embedding.  On the other hand, there are also approaches to eigenforms on the group $G$, but the moduli problem for $G$ is not representable.  For further discussion about the relationship between automorphic forms on these two spaces, see \cite[p. 2-3]{AIP2}.
\end{rmk}

Let $W$ denote the ring of Witt vectors associated to an algebraic closure of $\ZZ/p\ZZ$, and let $W_m = W/p^m W$.  We identify $W$ with the ring of integers in the maximal unramified extension of $\IQ_p$ inside an algebraic closure of $\IQ_p$.  We fix an embedding $K\hookrightarrow \bar{\IQ}_p$.  The image of $\OK$ under this embedding lies in $W$. 

The space of $p$-adic Hilbert modular forms is defined over the ordinary locus $\mathcal{M}^\ord$ (inside of $\mathcal{M}\times_{\Spec \OK}\Spec W$), which can be described as the nonvanishing locus of a lift of the Hasse invariant, like in \cite[Section 4.1.7]{hida}.  More precisely, the space of $p$-adic Hilbert modular forms is realized as follows.  We build an {\em Igusa tower} over $\mathcal{M}^\ord$ (as in, e.g, \cite[Section 8.1.1]{hida}).
 For each pair of positive integers $n, m$, $\Ig_{n,m}$ is defined to be a cover of $\mathcal{M}^\ord\times_W W_m$ classifying ordinary Hilbert--Blumenthal abelian varieties $A$ together with level $p^n$-structure $\mu_{p^n}\hookrightarrow A[p^n]$.  So we have canonical maps $\Ig_{n,m}\rightarrow \Ig_{n',m}$ for all $n'\geq n$ (and likewise for $m'\geq m$), giving us a tower of schemes.
Following the notation of \cite[Section 8.1.1]{hida},
we set
\begin{align*}
V_{n,m}&:=H^0\left(\Ig_{n,m}, \CO_{\Ig_{n,m}}\right)\\
V_{\infty, m} &:= \varinjlim_n V_{n,m}.
\end{align*}

Following \cite[Section 1.9]{kaCM} (or the more general discussion from \cite[Section 8.1.1]{hida}), the space of $p$-adic Hilbert modular forms is then
\begin{align*}
V:=V_{\infty, \infty}:=\varprojlim_m V_{\infty, m}.
\end{align*}

We identify $V_{\infty, \infty}$ with the ring of global sections of the structure sheaf of a formal scheme parametrizing Hilbert--Blumental abelian varieties with $p^\infty$-level structure.

An advantage of this construction is that is provides a canonical map from the space of Hilbert modular forms to the space $V$ of $p$-adic Hilbert modular forms (as in, e.g., \cite[Theorem (1.10.15)]{kaCM}).

\begin{rmk}
More generally, this construction can be modified to produce $p$-adic automorphic forms in other cases, such as in the setting of Shimura varieties of PEL type.  For a detailed treatment, see \cite{hida, CEFMV, EiMa}.
\end{rmk}

\subsubsection{$q$-expansion principles}

Like in Serre's construction, Deligne and Ribet's approach also relies substantially on properties of $q$-expansions of Eisenstein series.
So we now will need some $q$-expansion principles, i.e. theorems that explain to what degree Hilbert modular forms are determined by their $q$-expansions.  In Proposition \ref{alg-qexp} and Theorem \ref{cor-diff}, we choose the level structure so that the reduction of $\mathcal{M}$ is connected.  (Alternatively, we could modify the statements of Proposition \ref{alg-qexp} and Theorem \ref{thm-qpcoeff} to take a $q$-expansion at a cusp on each connected component.)

\begin{prop}[algebraic $q$-expansion principle for Hilbert modular forms]\label{alg-qexp}
Let $f$ be a Hilbert modular form defined over a ring $R$.
\begin{enumerate}
\item{If the algebraic $q$-expansion of $f$ vanishes at some cusp, then $f = 0$.}\label{qexp1}
\item{Let $R_0\subseteq R$ be a ring.  If the $q$-expansion of $f$ at some cusp has coefficients in $R_0$, then $f$ is defined over $R_0$.}\label{qexp2}
\end{enumerate}
\end{prop}
The proof of Statement \eqref{qexp1} relies on the irreducibility results in \cite{ribetirred, ra}, and Statement \eqref{qexp2} can be proved as a consequence of \eqref{qexp1} (similarly to the proof of \cite[Corollary 1.6.2]{ka2}).

The Fourier coefficients of the Eisenstein series needed for studying $L$-values of totally real Hecke characters have coefficients in the ring of integers $\CO$ of a number field.  So as a consequence of Proposition \ref{alg-qexp}\eqref{qexp2} and the fact that the algebraic and analytic $q$-expansions of a Hilbert modular form agree (by \cite[Equation (1.7.6)]{kaCM}), we have that our Eisenstein series are actually defined over $\CO$.

In order to construct the $p$-adic $L$-functions, we will also need a $p$-adic $q$-expansion principle for Hilbert modular forms.
\begin{thm}[$p$-adic $q$-expansion principle for Hilbert modular forms, (5.13) of \cite{DR}]\label{thm-qpcoeff}If $f\in V$ and the $q$-expansion of $f$ vanishes at some cusp, then $f = 0$.
Furthermore, if $R_0$ is flat over $\ZZ_p$, then the $R_0$-submodule $V_{R_0}$ of $p$-adic modular forms defined over $R_0$ consists of the elements $f\in V_{R_0}\otimes_{\ZZ_p} \IQ_p$ whose $q$-expansion coefficients lie in $R_0$, and if the $q$-expansion of a $p$-adic modular form $f$ at some cusp has coefficients in $R_0$, then the same is true at all the cusps.
\end{thm}

As an important consequence of Theorem \ref{thm-qpcoeff}, we obtain Corollary \ref{cor-diff}.

\begin{cor}[Corollary (5.14) of \cite{DR}]\label{cor-diff}
Let $f\in V_{R_0}\otimes\IQ_p$, and suppose that at some cusp, all the $q$-expansion coefficients, aside possibly from the constant term, of $f$ lie in $R_0$.  Then the difference between the constant terms of the $q$-expansions of $f$ at any two cusps also lies in $R_0$.
\end{cor}

\begin{proof}
Let $v\in V_{R_0}\otimes\IQ_p$ be such that at some cusp, all the $q$-expansion coefficients, aside possibly from the constant term, of $f$ lie in $R_0$.  Let $a$ be the constant term of $f$ at that cusp.  Then $a$ is a weight $0$ modular form, and all the coefficients of $f-a\in V_{R_0}\otimes \IQ_p$ lie in $R_0$.  So by Theorem \ref{thm-qpcoeff}, $f-a\in V_{R_0}$ and all the $q$-expansion coefficients of $f-a$, in particular its constant term, at any other cusp lie in $R_0$.  So the difference between any two constant terms of $f$ lies in $R_0$.
\end{proof}

As an immediate corollary of Corollary \ref{cor-diff}, we obtain:
\begin{cor}
If the abstract Kummer congruences hold for all the non-constant terms of the $q$-expansions of a family of $p$-adic modular forms $f$ at some cusp, then they also hold for the difference between the constant terms at two cusps.
\end{cor}

To prove Theorem \ref{DR-thm}, it then suffices to realize $L^{(p)}(1-k, \rho)$ in the constant term of a $q$-expansion of an Eisenstein series $E_{k, \rho}$ and observe that the constant term of $E_{k, \rho}$ at a cusp corresponding to a fractional ideal $\mathfrak{A}$ is $\rho(\mathfrak{A})\mathbf{N}(\mathfrak{A})^kL^{(p)}(1-k, \rho)$, which is proved in \cite[Theorem (6.1))]{DR}.

\subsection{The case where $\chi$ is a Hecke character of a CM field}\label{sec:cm}

Given that we just considered the case of Hecke characters of totally real fields, it is natural now to move to CM fields $\cmfield$.  Fix a CM type $\Sigma$ for $\cmfield$, i.e. a set of $[K:\IQ]/2$ embeddings $K\hookrightarrow \IC$ such that exactly one representative from each pair of complex conjugate embeddings $\left\{\sigma, \bar{\sigma}\right\}$ lies in $\Sigma$.  In \cite{kaCM}, Katz considered the case where $\chi: K^\times \backslash \adeles_K^\times\rightarrow \IC$ is a Hecke character of type $A_0$, i.e. $\chi$ is of the form
\begin{align*}
\chi = \chi_\fin\prod_{\sigma\in\Sigma}\left(\frac{1}{\sigma}\right)^k\left(\frac{\bar{\sigma}}{\sigma}\right)^{d(\sigma)},
\end{align*}
with $k$ a positive integer, $d(\sigma)$ a nonnegative integer for all $\sigma\in\Sigma$, and $\chi_\fin$ a finite order character.
Building on ideas of Eisenstein, the study of the algebraicity properties of the values $L(0, \chi)$ was initiated Damerell and later extended and completed by Goldstein--Schappacher \cite{GS1, GS2}, Shimura \cite{shimura-CM}, and Weil \cite{weil1}.  (A summary of the historical development is in \cite[\S5]{harder-schappacher}.)  

For $\chi$ of type $A_0$ as above, the values $L(0, \chi)$ can be expressed (in what is known as {\em Damerell's formula}) as finite sums of values of Eisenstein series in the space of Hilbert modular forms.  Thus, it is natural to try to construct Eisenstein measures suited to this application and adapt the techniques introduced thus far.

Indeed, this is what Katz did in \cite{kaCM}, but there are several new challenges Katz had to solve in this setting, which also helped uncover paths toward generalizations.  Because these challenges also arise more broadly, we continue the discussion in Section \ref{sec:general}.

\section{Generalizations and challenges}\label{sec:general}
Various constructions of automorphic $L$-functions are closely tied to Eisenstein series.  This includes Damerell's formula, the Rankin--Selberg method, and pullback methods like the doubling method. Each of these methods was used to prove algebraicity of certain values of the corresponding automorphic $L$-functions.  Given the developments discussed thus far, it is therefore natural to try to construct Eisenstein measures valued in appropriate spaces of $p$-adic automorphic forms and use those to construct $p$-adic $L$-functions.  Those familiar with any of these methods  might recall, though, that the Eisenstein series occurring in the constructions of the $L$-functions can be quite intricate (and likewise for computations of the Fourier coefficients), and furthermore, the $L$-functions are not simply realized as constant terms of these particular Eisenstein series.  

In addition, on the $p$-adic side, the slightest modification to input can have drastic geometric consequences.  For example, changing a prime from split to inert can lead to the entire ordinary locus employed in the definition of $p$-adic modular forms described above to disappear in certain settings.  In another direction, working with the full range of Hecke characters from Section \ref{sec:cm} requires considering Eisenstein series that are not holomorphic. 

Extending the approach of constructing Eisenstein measures to produce $p$-adic $L$-functions attached to Hecke characters of CM fields, as well as those considered in higher rank generalizations like \cite{apptoSHL, apptoSHLvv, HELS}, involves working in a setting where:
\begin{itemize}
\item{The approach of using constant terms (from \cite{serre, DR}) no longer applies, due to the fact that for Eisenstein series occurring in particular formulas for $L$-functions, the Fourier expansions of those Eisenstein series at cusps where it is convenient to work lack constant terms.  For example, the Fourier expansions of the particular Eisenstein series employed in the formulas in \cite{kaCM} turn out to lack constant terms at the cusps where they computed, as seen in, e.g., \cite[Theorem (3.2.3)]{kaCM}.  (That said, if one has a convenient way to compute and study the Fourier coefficients at a cusp where the constant term is nonzero, then this issue disappears.)}
\item{The Eisenstein series are substantially more complicated to construct.}
\item{The constructions of the $L$-functions require considering values of $\ci$ (not necessarily holomorphic) Eisenstein series.}
\item{The points in the ordinary locus needed in the construction of the $L$-functions might be empty.}
\end{itemize}
Moving beyond Hecke characters to Rankin--Selberg $L$-functions and $L$-functions associated to automorphic forms (e.g. through the doubling method), we also must content with the following:
\begin{itemize}
\item{$L$-functions might be represented not as finite sums of values of Eisenstein series, but instead as integrals of cusp form(s) against restrictions of  Eisenstein series to certain spaces (e.g. as in the doubling method)}
\end{itemize}

\subsection{Strategies of Katz, and beyond}\label{katzbeyond}

As noted in Section \ref{sec:cm}, Damerell's formula expresses values of the $L$-functions associated to Hecke characters of CM fields in terms of finite sums of values of Eisenstein series from the space of Hilbert modular forms.  In his construction of $p$-adic $L$-functions for CM fields \cite{kaCM}, Katz exploits the fact that the Eisenstein series get evaluated only at CM points, i.e. Hilbert--Blumenthal abelian varieties with complex multiplication.  He constructs an Eisenstein measure and then constructs a $p$-adic measure at each of these CM points $\mathfrak{A}$ by evaluating the Eisenstein series in the image of his Eisenstein measure at $\mathfrak{A}$.

Constructing the Eisenstein series and measure is considerably more involved than in the examples mentioned so far, though, and it is the subject of \cite[Chapter III and Section 4.2]{kaCM}.  Part of Katz's strategy is to introduce a {\em partial Fourier transform} (\cite[Section 3.1]{kaCM}), which allows him to construct an Eisenstein series amenable to computations for $L$-functions but which also has $q$-expansion coefficients that satisfy congruences (so that he can employ the $q$-expansion principles from above).  The key point with the partial Fourier transform is to take the Fourier transform of appropriate data that interpolates well to produce the Eisenstein series and then exploit the close relationship between the Fourier transform and the Fourier transform of the Fourier transform, namely that the Fourier transform of the Fourier transform of $t\mapsto f(t)$ is $t\mapsto f(-t)$. Hence we get an Eisenstein measure whose coefficients interpolate well.

To handle the $\ci$ Eisenstein series that occur in the construction of $L$-functions for CM fields, Katz must consider certain differential operators.  The $\ci$ Eisenstein series in the construction can be obtained by applying the Maass--Shimura differential operators to holomorphic Eisenstein series.  Katz exploits the Hodge theory of Hilbert--Blumenthal abelian varieties to construct $p$-adic analogues (built out of the Gauss--Manin connection and Kodaira--Spencer morphism) of those differential operators \cite[Chapter II]{kaCM}.  On $q$-expansions, these operators are a generalization of the operator $q\frac{d}{dq}$, and they preserve interpolation properties of the Hilbert modular forms to which they are applied. 

These techniques for constructing Eisenstein measures have since been extended to the PEL setting.  For example, differential operators on $p$-adic automorphic forms on unitary groups are the subject of \cite{EDiffOps, EFMV} (which also builds on \cite{hasv}), and they were used as a starting point in the construction of Eisenstein measures taking values in the space of $p$-adic automorphic forms on unitary groups in \cite{apptoSHL, apptoSHLvv}, which were in turn employed in the constructions of $p$-adic $L$-functions in \cite{HELS, EW}.

Like in Section \ref{sec:pHMF}, Katz's construction is over the ordinary locus.  This introduces a serious obstacle, namely that there are no ordinary CM points, if $p$ is inert.  Given that Damerell's formula is a sum over CM points, this means Katz's approach did not address $p$ inert.

Over four decades passed before an approach to $p$ inert was introduced.  In \cite{andreatta-iovita}, Andreatta and Iovita explain how to adapt Katz's approach to the case of quadratic imaginary fields with $p$ inert.  In separate work \cite{kriz2018new}, Kriz also introduced an approach for inert $p$.  Parts of \cite{andreatta-iovita} are also being extended to the case of CM fields in \cite{aycock, graziani2020modular}.  The idea of Andreatta and Iovita is to work instead with {\em overconvergent} $p$-adic modular forms and modify the approach to handling the differential operators.  Whereas Katz exploits Dwork's {\em unit root splitting} that exists over the ordinary locus, Andreatta and Iovita build an operator from the Gauss--Manin connection and then take pairings that do not require projecting modulo a unit root splitting.

\subsection{Working with pairings and pullback methods}\label{sec:pairing}

Katz's approach to constructing Eisenstein measures provides a starting point for other cases, in particular automorphic forms in the PEL setting.  Since we are often faced with representations of $L$-functions not as a finite sum but rather as an integral of an Eisenstein series against cusp form(s), we now briefly explain the key ideas for adapting such a representation to the $p$-adic setting.  We discuss this strategy in the context of the Rankin--Selberg zeta function, where it was first developed (by Hida in \cite{hi85}), but it has also since been extended to various settings, including, among others, in \cite{hidaLfcn, pa, liuJussieu, liu-rosso, HELS}.

The Rankin--Selberg product of a weight $k$ holomorphic cusp form $f = \sum_{n\geq 1}a_n q^n$ and a weight $\ell\leq k$ holomorphic modular form $g = \sum_{n\geq 0}b_n q^n$ is a zeta series
\begin{align*}
D(s, f, g) = \sum_{n=1}^\infty \frac{a_n b_n}{n^s}.
\end{align*}
Shimura and Rankin proved in \cite{shimura-RS, rankin} that 
\begin{align*}
D(k-1-r, f, g) = c\pi^l\langle\tilde{f}, g\delta_\lambda^{(r)} E\rangle,
\end{align*}
 where $E$ denotes a particular weight $\lambda:=k-\ell-2r$ Eisenstein series, $\tilde{f}(z):=\overline{f\left(-\bar{z}\right)},$ $\delta_\lambda^{(r)}$ is a Maass--Shimura operator that raises the weight of a modular form of weight $\lambda$ by $2r$ (so $\delta_\lambda^{(r)}:=\partial_{\lambda+2r-2}\circ \partial_{\lambda+2}\circ\partial_\lambda$ with $\delta_\lambda:=\frac{1}{2\pi i}\left(\frac{\lambda}{2iy}+\frac{\partial}{\partial z}\right)$), $c=\frac{\Gamma(k-\ell-2r)}{\Gamma(k-1-r)\Gamma(k-\ell-r)}\frac{(-1)^r 4^{k-1}N}{3}\prod_{p\divides N}\left(1+p^{-1}\right)$ (with $N$ the level of the modular forms), and $\langle, \rangle$ denotes the Petersson inner product.  As a consequence, Shimura proved in \cite[Theorem 2]{shimura-RS} that
 \begin{align*}
 \frac{\pi^{-k}D(m, f, g)}{\langle f, f\rangle}
 \end{align*}
 is algebraic for all integers $k, \ell, m$ satisfying $\ell<k$ and $\frac{k+\ell-2}{2}<m<k$.
 
In \cite{hi85}, Hida constructed $p$-adic Rankin--Selberg zeta functions by building on Shimura's approach to studying algebraicity.  In particular, the idea to interpret the Rankin--Selberg zeta function in terms of the Peterssen pairing plays a key role, and this remains true in extensions to higher rank groups (including in the discussions of algebraicity in \cite{hasv} and in extensions to the $p$-adic case in PEL settings involving the doubling method in \cite{liuJussieu, liu-rosso, HELS}).  The idea is to reinterpret the linear Petersson pairing $\langle h_1, h_2\rangle$ as a functional $\ell_{h_1}\left(h_2\right)$.  This suggests identifying a space of modular forms with its dual space, which in turn leads to use of the associated Hecke algebra.  This is the point that allows Hida to integrate Eisenstein measures (generally coming from familiar families of Eisenstein series, at least in the case of modular forms) into the construction of $p$-adic $L$-functions.  While this approach makes sense in higher rank (e.g. in the context of the doubling method), putting it into practice is nontrivial for various reasons, including geometric issues (like those mentioned above) and new properties of the Hecke algebra that must be taken into account.

\subsection{Some remaining challenges and future directions}
Putting aside the bigger goal of proving the Greenberg--Iwasawa main conjectures, challenges still remain for producing $p$-adic $L$-functions.  Even in settings where we have constructions of $L$-functions closely tied to the behavior of Eisenstein series and we anticipate the existence of Eisenstein measures, actually carrying out the construction can be nontrivial.  We conclude by highlighting three categories of challenges and suggest some future directions toward resolving them:
\begin{enumerate}
\item{As noted in Section \ref{sec:pairing}, the pairings that arise from integral representations of $L$-functions can be useful for $p$-adic interpolation, but one often has to deal with significant technical challenges.  Properties of Hecke algebras (and the {\em ordinary} Hecke algbras where one often works in practice) can present obstacles.  For example, a Gorenstein property is often useful, but not necessarily known, in this context.  Pilloni's higher Hida theory seems to present a promising and natural alternative framework for interpreting these pairings \cite{pilloni-higher, loeffler2019higher, rboxer2020higher}.}
\item{As noted in Section \ref{katzbeyond}, cases where the prime $p$ does not split can lead to considerable geometric challenges, which have been recently addressed in low rank in \cite{andreatta-iovita}.  For unitary groups, work on differential operators in \cite{EiMa, DSG, DSG2} and Hecke operators in \cite{brasca-rosso} addresses some challenges that arise when the ordinary locus is empty, but work remains in the inert case (even just for constructing appropriate Eisenstein series for the Eisenstein measure) to construct the full $p$-adic $L$-functions.  In one of the most promising directions, the work in \cite{aycock, graziani2020modular} suggests the possibility of extending the techniques of \cite{andreatta-iovita} to the PEL setting, but again, details of the Eisenstein measures would still need to be worked out by adjusting the choices of local data that feed into the partial Fourier transforms in \cite{apptoSHL}.}
\item{At a more fundamental level, before one can construct $p$-adic $L$-functions via the method of Eisenstein measures, one needs a representation of the $L$-function in terms of Eisenstein series.  Such a representation, though, is insufficient unless we also can reinterpret it algebraically.  For example, given the success in adapting the doubling method to the $p$-adic setting in \cite{liuJussieu, EW, HELS}, it is natural to try to adapt the {\em twisted doubling} representation of $L$-functions (i.e. for producing $L$-functions associated to a twist of a cuspidal automorphic representation by a representation of $\GL_n$ for some $n$) in \cite{CFGK} to the $p$-adic setting.  As of yet, though, we do not have an appropriate interpretation in terms of algebraic geometry or another familiar algebraic tool, and without an algebraic interpretation, are unlikely to see a path toward a $p$-adic realization.  There is currently active work to produce integral representations of various $L$-functions.  It will be interesting to see which ones become suitable for proving algebraicity results, either in terms of the techniques described above or in terms of those yet to be discovered.  
}
\end{enumerate}

\bibliography{../eisch} 

\newcommand{\etalchar}[1]{$^{#1}$}
\providecommand{\bysame}{\leavevmode\hbox to3em{\hrulefill}\thinspace}
\providecommand{\MR}{\relax\ifhmode\unskip\space\fi MR }
\providecommand{\MRhref}[2]{%
  \href{http://www.ams.org/mathscinet-getitem?mr=#1}{#2}
}
\providecommand{\href}[2]{#2}
\begin{thebibliography}{EFMV18}

\bibitem[AI19]{andreatta-iovita}
Fabrizio Andreatta and Adrian Iovita, \emph{Katz type $p$-adic {$L$}-functions
  for primes $p$ non-split in the {CM} field}, Preprint. Available at
  \url{https://arxiv.org/pdf/1905.00792.pdf}.

\bibitem[AIP16]{AIP2}
Fabrizio Andreatta, Adrian Iovita, and Vincent Pilloni, \emph{On overconvergent
  {H}ilbert modular cusp forms}, Ast\'{e}risque (2016), no.~382, 163--193.
  \MR{3581177}

\bibitem[Ayc21]{aycock}
Jon Aycock, \emph{Families of differential operators for overconvergent
  {H}ilbert modular forms}, 2021, In preparation.

\bibitem[Bar78]{barsky}
Daniel Barsky, \emph{Fonctions zeta {$p$}-adiques d'une classe de rayon des
  corps de nombres totalement r\'{e}els}, Groupe d'{E}tude d'{A}nalyse
  {U}ltram\'{e}trique (5e ann\'{e}e: 1977/78), Secr\'{e}tariat Math., Paris,
  1978, pp.~Exp. No. 16, 23. \MR{525346}

\bibitem[BCG20]{BCG}
Nicolas Bergeron, Pierre Charollois, and Luis~E. Garcia, \emph{Transgressions
  of the {E}uler class and {E}isenstein cohomology of {${\rm GL}_N({\bf Z})$}},
  Jpn. J. Math. \textbf{15} (2020), no.~2, 311--379. \MR{4120422}

\bibitem[Beh09]{beh}
Mark Behrens, \emph{Eisenstein orientation}, Typed notes available at
  \url{http://www-math.mit.edu/~mbehrens/other/coredump.pdf}.

\bibitem[Bou98]{bourbaki1998commutative}
Nicolas Bourbaki, \emph{Commutative algebra. {C}hapters 1--7}, Elements of
  Mathematics (Berlin), Springer-Verlag, Berlin, 1998, Translated from the
  French, Reprint of the 1989 English translation. \MR{1727221}

\bibitem[BP20]{rboxer2020higher}
George Boxer and Vincent Pilloni, \emph{Higher {H}ida and {C}oleman theories on
  the modular curve}, 2020, Preprint available at
  \url{https://arxiv.org/pdf/2002.06845.pdf}.

\bibitem[BR19]{brasca-rosso}
R.~{Brasca} and G.~{Rosso}, \emph{{Hida theory over some unitary Shimura
  varieties without ordinary locus}}, Amer. J. Math. (2019), Accepted for
  publication. Preprint available at
  \url{https://drive.google.com/file/d/100MYaAxzmDATss6sxGoQkb5SBKu0huQI/view}.

\bibitem[CD14]{Ch-Da}
Pierre Charollois and Samit Dasgupta, \emph{Integral {E}isenstein cocycles on
  {${\bf GL}_n$}, {I}: {S}czech's cocycle and {$p$}-adic {$L$}-functions of
  totally real fields}, Camb. J. Math. \textbf{2} (2014), no.~1, 49--90.
  \MR{3272012}

\bibitem[CEF{\etalchar{+}}16]{CEFMV}
Ana Caraiani, Ellen Eischen, Jessica Fintzen, Elena Mantovan, and Ila Varma,
  \emph{{$p$}-adic {$q$}-expansion principles on unitary {S}himura varieties},
  Directions in number theory, vol.~3, Springer, [Cham], 2016, pp.~197--243.
  \MR{3596581}

\bibitem[CFGK19]{CFGK}
Yuanqing Cai, Solomon Friedberg, David Ginzburg, and Eyal Kaplan,
  \emph{Doubling constructions and tensor product {$L$}-functions: the linear
  case}, Invent. Math. \textbf{217} (2019), no.~3, 985--1068. \MR{3989257}

\bibitem[Cla40]{clausen}
Thomas Clausen, \emph{Theorem}, Astronomische Nachrichten \textbf{17} (1840),
  no.~22, 351--352.

\bibitem[CN79]{cassou-nogues}
Pierrette Cassou-Nogu\`es, \emph{Valeurs aux entiers n\'{e}gatifs des fonctions
  z\^{e}ta et fonctions z\^{e}ta {$p$}-adiques}, Invent. Math. \textbf{51}
  (1979), no.~1, 29--59. \MR{524276}

\bibitem[Coa89]{coatesII}
John Coates, \emph{On {$p$}-adic {$L$}-functions attached to motives over
  {${\bf Q}$}. {II}}, Bol. Soc. Brasil. Mat. (N.S.) \textbf{20} (1989), no.~1,
  101--112. \MR{1129081 (92j:11060b)}

\bibitem[CPR89]{CoPR}
John Coates and Bernadette Perrin-Riou, \emph{On {$p$}-adic {$L$}-functions
  attached to motives over {${\bf Q}$}}, Algebraic number theory, Adv. Stud.
  Pure Math., vol.~17, Academic Press, Boston, MA, 1989, pp.~23--54.
  \MR{1097608 (92j:11060a)}

\bibitem[CS74]{coates-sinnott}
J.~Coates and W.~Sinnott, \emph{On {$p$}-adic {$L$}-functions over real
  quadratic fields}, Invent. Math. \textbf{25} (1974), 253--279. \MR{354615}

\bibitem[DR80]{DR}
Pierre Deligne and Kenneth~A. Ribet, \emph{Values of abelian {$L$}-functions at
  negative integers over totally real fields}, Invent. Math. \textbf{59}
  (1980), no.~3, 227--286. \MR{579702 (81m:12019)}

\bibitem[dSG16]{DSG}
Ehud de~Shalit and Eyal~Z. Goren, \emph{A theta operator on {P}icard modular
  forms modulo an inert prime}, Res. Math. Sci. \textbf{3} (2016), Paper No.
  28, 65. \MR{3543240}

\bibitem[dSG19]{DSG2}
\bysame, \emph{Theta operators on unitary {S}himura varieties}, Algebra Number
  Theory \textbf{13} (2019), no.~8, 1829--1877. \MR{4017536}

\bibitem[EFMV18]{EFMV}
Ellen Eischen, Jessica Fintzen, Elena Mantovan, and Ila Varma,
  \emph{Differential operators and families of automorphic forms on unitary
  groups of arbitrary signature}, Doc. Math. \textbf{23} (2018), 445--495.
  \MR{3846052}

\bibitem[EHLS20]{HELS}
Ellen Eischen, Michael Harris, Jianshu Li, and Christopher Skinner,
  \emph{{$p$}-adic {$L$}-functions for unitary groups}, Forum Math. Pi
  \textbf{8} (2020), e9, 160. \MR{4096618}

\bibitem[Eis12]{EDiffOps}
Ellen~E. Eischen, \emph{{$p$}-adic differential operators on automorphic forms
  on unitary groups}, Ann. Inst. Fourier (Grenoble) \textbf{62} (2012), no.~1,
  177--243. \MR{2986270}

\bibitem[Eis14]{apptoSHLvv}
Ellen Eischen, \emph{A $p$-adic {E}isenstein measure for vector-weight
  automorphic forms}, Algebra Number Theory \textbf{8} (2014), no.~10,
  2433--2469. \MR{3298545}

\bibitem[Eis15]{apptoSHL}
Ellen~E. Eischen, \emph{A $p$-adic {E}isenstein measure for unitary groups}, J.
  Reine Angew. Math. \textbf{699} (2015), 111--142. \MR{3305922}

\bibitem[EM21]{EiMa}
Ellen Eischen and Elena Mantovan, \emph{$p$-adic families of automorphic forms
  in the $\mu$-ordinary setting}, Amer. J. Math. \textbf{143} (2021), no.~1,
  1--51.

\bibitem[EW16]{EW}
Ellen Eischen and Xin Wan, \emph{{$p$}-adic {E}isenstein series and
  {$L$}-functions of certain cusp forms on definite unitary groups}, J. Inst.
  Math. Jussieu \textbf{15} (2016), no.~3, 471--510. \MR{3505656}

\bibitem[Gir90]{girstmair}
Kurt Girstmair, \emph{A theorem on the numerators of the {B}ernoulli numbers},
  Amer. Math. Monthly \textbf{97} (1990), no.~2, 136--138. \MR{1041890}

\bibitem[Gor02]{goren}
Eyal~Z. Goren, \emph{Lectures on {H}ilbert modular varieties and modular
  forms}, CRM Monograph Series, vol.~14, American Mathematical Society,
  Providence, RI, 2002, With the assistance of Marc-Hubert Nicole. \MR{1863355}

\bibitem[Gra20]{graziani2020modular}
Giacomo Graziani, \emph{Modular sheaves of de {R}ham classes on {H}ilbert
  formal modular schemes for unramified primes}, 2020, arXiv preprint available
  at \url{https://arxiv.org/pdf/2007.15997.pdf}.

\bibitem[Gre89]{green3}
Ralph Greenberg, \emph{Iwasawa theory for {$p$}-adic representations},
  Algebraic number theory, Adv. Stud. Pure Math., vol.~17, Academic Press,
  Boston, MA, 1989, pp.~97--137. \MR{1097613 (92c:11116)}

\bibitem[Gre91]{green2}
\bysame, \emph{Iwasawa theory for motives}, {$L$}-functions and arithmetic
  ({D}urham, 1989), London Math. Soc. Lecture Note Ser., vol. 153, Cambridge
  Univ. Press, Cambridge, 1991, pp.~211--233. \MR{1110394 (93a:11092)}

\bibitem[Gre94]{green1}
\bysame, \emph{Iwasawa theory and {$p$}-adic deformations of motives}, Motives
  ({S}eattle, {WA}, 1991), Proc. Sympos. Pure Math., vol.~55, Amer. Math. Soc.,
  Providence, RI, 1994, pp.~193--223. \MR{1265554 (95i:11053)}

\bibitem[GS81]{GS1}
Catherine Goldstein and Norbert Schappacher, \emph{S\'{e}ries d'{E}isenstein et
  fonctions {$L$} de courbes elliptiques \`a multiplication complexe}, J. Reine
  Angew. Math. \textbf{327} (1981), 184--218. \MR{631315}

\bibitem[GS83]{GS2}
\bysame, \emph{Conjecture de {D}eligne et {$\Gamma $}-hypoth\`ese de
  {L}ichtenbaum sur les corps quadratiques imaginaires}, C. R. Acad. Sci. Paris
  S\'{e}r. I Math. \textbf{296} (1983), no.~15, 615--618. \MR{705673}

\bibitem[Har81]{hasv}
Michael Harris, \emph{Special values of zeta functions attached to {S}iegel
  modular forms}, Ann. Sci. \'Ecole Norm. Sup. (4) \textbf{14} (1981), no.~1,
  77--120. \MR{MR618732 (82m:10046)}

\bibitem[Hen05]{hensel}
K.~Hensel, \emph{\"{U}ber eine neue {B}egr\"{u}ndung der {T}heorie der
  algebraischen {Z}ahlen}, J. Reine Angew. Math. \textbf{128} (1905), 1--32.
  \MR{1580642}

\bibitem[Hid85]{hi85}
Haruzo Hida, \emph{A {$p$}-adic measure attached to the zeta functions
  associated with two elliptic modular forms. {I}}, Invent. Math. \textbf{79}
  (1985), no.~1, 159--195. \MR{MR774534 (86m:11097)}

\bibitem[Hid91]{hidaLfcn}
\bysame, \emph{On {$p$}-adic {$L$}-functions of {${\rm GL}(2)\times {\rm
  GL}(2)$} over totally real fields}, Ann. Inst. Fourier (Grenoble) \textbf{41}
  (1991), no.~2, 311--391. \MR{1137290 (93b:11052)}

\bibitem[Hid04]{hida}
\bysame, \emph{{$p$}-adic automorphic forms on {S}himura varieties}, Springer
  Monographs in Mathematics, Springer-Verlag, New York, 2004. \MR{MR2055355
  (2005e:11054)}

\bibitem[Hop95]{hopkins94}
Michael~J. Hopkins, \emph{Topological modular forms, the {W}itten genus, and
  the theorem of the cube}, Proceedings of the {I}nternational {C}ongress of
  {M}athematicians, {V}ol.\ 1, 2 ({Z}\"urich, 1994) (Basel), Birkh\"auser,
  1995, pp.~554--565. \MR{1403956 (97i:11043)}

\bibitem[Hop02]{hopkinsICM}
M.~J. Hopkins, \emph{Algebraic topology and modular forms}, Proceedings of the
  {I}nternational {C}ongress of {M}athematicians, {V}ol. {I} ({B}eijing, 2002)
  (Beijing), Higher Ed. Press, 2002, pp.~291--317. \MR{1989190 (2004g:11032)}

\bibitem[HS85]{harder-schappacher}
G.~Harder and N.~Schappacher, \emph{Special values of {H}ecke {$L$}-functions
  and abelian integrals}, Workshop {B}onn 1984 ({B}onn, 1984), Lecture Notes in
  Math., vol. 1111, Springer, Berlin, 1985, pp.~17--49. \MR{797414}

\bibitem[IO06]{IO}
Taku Ishii and Takayuki Oda, \emph{A short history on investigation of the
  special values of zeta and {$L$}-functions of totally real number fields},
  Automorphic forms and zeta functions, World Sci. Publ., Hackensack, NJ, 2006,
  pp.~198--233. \MR{2208776}

\bibitem[Iwa69a]{iw}
Kenkichi Iwasawa, \emph{Analogies between number fields and function fields},
  Some {R}ecent {A}dvances in the {B}asic {S}ciences, {V}ol. 2 ({P}roc.
  {A}nnual {S}ci. {C}onf., {B}elfer {G}rad. {S}chool {S}ci., {Y}eshiva {U}niv.,
  {N}ew {Y}ork, 1965-1966), Belfer Graduate School of Science, Yeshiva Univ.,
  New York, 1969, pp.~203--208. \MR{0255510 (41 \#172)}

\bibitem[Iwa69b]{iw2}
\bysame, \emph{On {$p$}-adic {$L$}-functions}, Ann. of Math. (2) \textbf{89}
  (1969), 198--205. \MR{269627}

\bibitem[Kat73]{ka2}
Nicholas~M. Katz, \emph{{$p$}-adic properties of modular schemes and modular
  forms}, Modular functions of one variable, {III} ({P}roc. {I}nternat.
  {S}ummer {S}chool, {U}niv. {A}ntwerp, {A}ntwerp, 1972), Springer, Berlin,
  1973, pp.~69--190. Lecture Notes in Mathematics, Vol. 350. \MR{MR0447119 (56
  \#5434)}

\bibitem[Kat78]{kaCM}
\bysame, \emph{{$p$}-adic {$L$}-functions for {CM} fields}, Invent. Math.
  \textbf{49} (1978), no.~3, 199--297. \MR{MR513095 (80h:10039)}

\bibitem[KL64]{KL}
Tomio Kubota and Heinrich-Wolfgang Leopoldt, \emph{Eine {$p$}-adische {T}heorie
  der {Z}etawerte. {I}. {E}inf\"uhrung der {$p$}-adischen {D}irichletschen
  {$L$}-{F}unktionen}, J. Reine Angew. Math. \textbf{214/215} (1964), 328--339.
  \MR{0163900 (29 \#1199)}

\bibitem[Kli62]{klingen}
Helmut Klingen, \emph{\"{U}ber die {W}erte der {D}edekindschen {Z}etafunktion},
  Math. Ann. \textbf{145} (1961/62), 265--272. \MR{133304}

\bibitem[Kri18]{kriz2018new}
Daniel Kriz, \emph{A new $p$-adic {M}aass--{S}himura operator and supersingular
  {R}ankin--{S}elberg $p$-adic {$L$}-functions}, 2018, arXiv preprint available
  at \url{https://arxiv.org/pdf/1805.03605.pdf}.

\bibitem[Kum50a]{kummer36}
E.~E. Kummer, \emph{Allgemeiner {B}eweis des {F}ermatschen {S}atzes, da\ss die
  {G}leichung {$x^\lambda+y^\lambda=z^\lambda$} durch ganze {Z}ahlen
  unl\"{o}sbar ist, f\"{u}r alle diejenigen {P}otenz-{E}xponenten {$\lambda$}
  welche ungerade {P}rimzahlen sind und in den {Z}\"{a}hlern der ersten
  {$1/2(\lambda)$} {B}ernoullischen zahlen als {F}actoren nicht vorkommen}, J.
  Reine Angew. Math. \textbf{40} (1850), 130--138. \MR{1578681}

\bibitem[Kum50b]{kummer34}
\bysame, \emph{Bestimmung der {A}nzahl nicht \"{a}quivalenter {C}lassen f\"{u}r
  die aus {$\lambda^{ten}$} {W}urzeln der {E}inheit gebildeten complexen
  {Z}ahlen und die idealen {F}actoren derselben}, J. Reine Angew. Math.
  \textbf{40} (1850), 93--116. \MR{1578679}

\bibitem[Kum51]{kummer}
\bysame, \emph{\"{U}ber eine allgemeine {E}igenschaft der rationalen
  {E}ntwickelungscoefficienten einer bestimmten {G}attung analytischer
  {F}unctionen}, J. Reine Angew. Math. \textbf{41} (1851), 368--372.
  \MR{1578727}

\bibitem[Liu20]{liuJussieu}
Zheng Liu, \emph{{$p$}-adic {$L$}-functions for ordinary families on symplectic
  groups}, J. Inst. Math. Jussieu \textbf{19} (2020), no.~4, 1287--1347.
  \MR{4120810}

\bibitem[LPSZ19]{loeffler2019higher}
David Loeffler, Vincent Pilloni, Christopher Skinner, and Sarah~Livia Zerbes,
  \emph{Higher hida theory and $p$-adic l-functions for {GS}p(4)}, 2019,
  Preprint available at \url{https://arxiv.org/pdf/1905.08779.pdf}.

\bibitem[LR20]{liu-rosso}
Zheng Liu and Giovanni Rosso, \emph{Non-cuspidal {H}ida theory for {S}iegel
  modular forms and trivial zeros of $p$-adic {$L$}-functions}, Math. Ann.
  \textbf{378} (2020), no.~1-2, 153--231. \MR{4150915}

\bibitem[MSD74]{MSD}
B.~Mazur and P.~Swinnerton-Dyer, \emph{Arithmetic of {W}eil curves}, Invent.
  Math. \textbf{25} (1974), 1--61. \MR{0354674 (50 \#7152)}

\bibitem[MW84]{MW}
B.~Mazur and A.~Wiles, \emph{Class fields of abelian extensions of {${\bf
  Q}$}}, Invent. Math. \textbf{76} (1984), no.~2, 179--330. \MR{742853
  (85m:11069)}

\bibitem[Pan03]{pa}
A.~A. Panchishkin, \emph{Two variable {$p$}-adic {$L$} functions attached to
  eigenfamilies of positive slope}, Invent. Math. \textbf{154} (2003), no.~3,
  551--615. \MR{MR2018785 (2004k:11065)}

\bibitem[Pil20]{pilloni-higher}
Vincent Pilloni, \emph{Higher coherent cohomology and $p$-adic modular forms of
  singular weight}, Duke Math. J. (2020), To appear.

\bibitem[Ran52]{rankin}
R.~A. Rankin, \emph{The scalar product of modular forms}, Proc. London Math.
  Soc. (3) \textbf{2} (1952), 198--217. \MR{49231}

\bibitem[Rap78]{ra}
M.~Rapoport, \emph{Compactifications de l'espace de modules de
  {H}ilbert-{B}lumenthal}, Compositio Math. \textbf{36} (1978), no.~3,
  255--335. \MR{MR515050 (80j:14009)}

\bibitem[Rib75]{ribetirred}
Kenneth~A. Ribet, \emph{{$p$}-adic interpolation via {H}ilbert modular forms},
  Algebraic geometry ({P}roc. {S}ympos. {P}ure {M}ath., {V}ol. 29, {H}umboldt
  {S}tate {U}niv., {A}rcata, {C}alif., 1974), Amer. Math. Soc., Providence, R.
  I., 1975, pp.~581--592. \MR{0419414 (54 \#7435)}

\bibitem[Ser73]{serre}
Jean-Pierre Serre, \emph{Formes modulaires et fonctions z\^eta {$p$}-adiques},
  Modular functions of one variable, {III} ({P}roc. {I}nternat. {S}ummer
  {S}chool, {U}niv. {A}ntwerp, 1972), Springer, Berlin, 1973, pp.~191--268.
  Lecture Notes in Math., Vol. 350. \MR{MR0404145 (53 \#7949a)}

\bibitem[Shi75]{shimura-CM}
Goro Shimura, \emph{On some arithmetic properties of modular forms of one and
  several variables}, Ann. of Math. (2) \textbf{102} (1975), no.~3, 491--515.
  \MR{491519}

\bibitem[Shi76]{shimura-RS}
\bysame, \emph{The special values of the zeta functions associated with cusp
  forms}, Comm. Pure Appl. Math. \textbf{29} (1976), no.~6, 783--804.
  \MR{434962}

\bibitem[Shi00]{shar}
\bysame, \emph{Arithmeticity in the theory of automorphic forms}, Mathematical
  Surveys and Monographs, vol.~82, American Mathematical Society, Providence,
  RI, 2000. \MR{MR1780262 (2001k:11086)}

\bibitem[Sie69]{siegel1}
Carl~Ludwig Siegel, \emph{Berechnung von {Z}etafunktionen an ganzzahligen
  {S}tellen}, Nachr. Akad. Wiss. G\"{o}ttingen Math.-Phys. Kl. II \textbf{1969}
  (1969), 87--102. \MR{0252349}

\bibitem[Sie70]{siegel2}
\bysame, \emph{\"{U}ber die {F}ourierschen {K}oeffizienten von {M}odulformen},
  Nachr. Akad. Wiss. G\"{o}ttingen Math.-Phys. Kl. II \textbf{1970} (1970),
  15--56. \MR{0285488}

\bibitem[vS40]{staudt}
K.~G.~C. von Staudt, \emph{Beweis eines {L}ehrsatzes, die {B}ernoullischen
  {Z}ahlen betreffen}, J. Reine Angew. Math. \textbf{21} (1840), 372--374.
  \MR{1578267}

\bibitem[vS45]{vstaudt}
\bysame, \emph{De numeris {B}ernoullianis}, Universitatsschrift, Erlangen
  (1845).

\bibitem[Was97]{wa}
Lawrence~C. Washington, \emph{Introduction to cyclotomic fields}, second ed.,
  Graduate Texts in Mathematics, vol.~83, Springer-Verlag, New York, 1997.
  \MR{MR1421575 (97h:11130)}

\bibitem[Wei99]{weil1}
Andr{\'e} Weil, \emph{Elliptic functions according to {E}isenstein and
  {K}ronecker}, Classics in Mathematics, Springer-Verlag, Berlin, 1999, Reprint
  of the 1976 original. \MR{1723749 (2000g:11001)}

\end{thebibliography}

\end{document}